\documentclass[11pt,a4paper]{article}
\usepackage[latin1]{inputenc}
\usepackage{indentfirst}
\usepackage{amsmath}
\usepackage{amsfonts}
\usepackage{amsthm}
\usepackage{amssymb}
\usepackage{graphics}
\usepackage{graphicx}
\usepackage{multicol}

\def\i{\mathbf{g}}
\def\u{\mathbf{g}}
\def\f{\mathbf{f}}

\def\F{\mathcal{F}}
\def\B{\mathcal{B}}

\def\Z{\mathbb{Z}}
\def\N{\mathbb{N}}

\def\P{\mathbb{P}}
\def\E{\mathbb{E}}

\def\YY{\Upsilon}

\def\RRR{\mathbb{R}}

\def\t{\textrm}
\def\d{\textrm{d}}
\def\w{\widetilde}

\def\w{\tilde}

\newcommand{\be} {\begin{equation}}
\newcommand{\ee} {\end{equation}}
\newcommand{\bea} {\begin{eqnarray}}
\newcommand{\eea} {\end{eqnarray}}
\newcommand{\Bea} {\begin{eqnarray*}}
\newcommand{\Eea} {\end{eqnarray*}}

\newtheorem{Thm}{Theorem}
\newtheorem{Lem}[Thm]{Lemma}
\newtheorem{Pte}[Thm]{Proposition}
\newtheorem{Cor}[Thm]{Corollary}
\theoremstyle{definition} 
\theoremstyle{definition} \newtheorem*{key}{Key words}
\theoremstyle{definition} \newtheorem*{ams}{A.M.S. Classification}

\theoremstyle{remark}
\begin{document}
\title{Surviving particles  for subcritical branching processes in random environment}
\author{Vincent Bansaye \footnote{Laboratoire de Probabilités et Modèles Aléatoires. Université Pierre et Marie Curie et C.N.R.S. UMR 7599. 175, rue du
Chevaleret, 75 013 Paris, France.
$\newline$
\emph{e-mail} : vincent.bansaye@upmc.fr }}
\maketitle \vspace{3cm}
\begin{abstract}
 The asymptotic behavior of a subcritical Branching Process in Random Environment (BPRE)
 starting with several particles depends on whether the BPRE is
strongly subcritical (SS), intermediate subcritical
(IS) or weakly subcritical (WS). %Descendances of particles for BPRE are not independent.
In the (SS+IS) case, the asymptotic probability of survival is proportional to the initial
 number of
particles, and conditionally on the
survival of the population, only one initial particle survives $a.s.$ These two properties
 do not hold in the (WS) case and different asymptotics are established, which require  new results on random walks with negative drift. We
 provide an interpretation
of these results by characterizing the sequence of environments
selected  when we condition on the survival of particles. This also
raises  the problem of the dependence of the Yaglom quasistationary
distributions on the initial number of particles and the asymptotic
behavior of the Q-process associated with a subcritical BPRE.
\end{abstract}
%\begin{key}
%\end{key}
%\begin{ams}
%\end{ams}
\begin{key}
Branching process in random environment (BPRE). Yaglom distribution. Q-process. Random walk with negative drift.
\end{key}
\begin{ams}
60J80, 60J85, 60K37.
\end{ams}

\section{Introduction}
Let $f$ be the generating function of a random probability measure on $\N$ and $(f_n)_{n\in\N}$ a sequence
of iid copies of $f$ which serve as random environment. We consider a Branching Process in Random Environment (BPRE) $(Z_n)_{n\in\N}$  induced by $(f_n)_{n\in\N}$  \cite{afa, slide,  at,   bpree, bpre}. This means that conditionally on the environment
 $(f_n)_{n\in\N}$, particles at generation $n$ reproduce
independently of each other and their offsprings have generating function $f_n$. 

We can think of  a population of plants which have a one year life-cycle. Each year the weather conditions (the environment) vary,  which impacts the reproductive success of the plant. Given the climate, all the plants reproduce according to the same  mechanism.

Then $Z_n$
is the number of particles at generation $n$
and $Z_{n+1}$ is the sum of $Z_n$ independent random variables with generating function $f_n$.
That is, for every $n\in\N$,
$$\E\big(s^{Z_{n+1}}\vert Z_0,\dots,Z_n; \ f_0,\dots,f_n\big)=f_n(s)^{Z_n} \qquad (0\leq s\leq 1).$$
In the whole paper, we denote by $\P_k$ the probability associated with $k$ initial particles and  $F_n:=f_{0}\circ \cdots\circ f_{n-1}$. Then,  we have for every $k\in\N$,
$$\E_k(s^{Z_{n+1}}\ \vert \ \ f_0, ..., \ f_n)=
F_{n+1}(s)^k \qquad (0\leq s\leq 1).$$

When the environment is deterministic ($i.e.$ $f$ is a deterministic generating function), this process is the Galton Watson
process (GW) and  $f$ is the generating function of the reproduction law. \\

In this  paper, we consider the subcritical case  :
$$\E\big(\t{log}(f'(1))\big)<0.$$
Then extinction occurs $a.s.$, that is
$$\P(\exists n\in\N : Z_n=0)=1. $$
For a subcritical GW process, if further $\E(Z_1\log^+(Z_1))<\infty$, then there exists $c>0$ such that
$\P(Z_n>0)\sim c f'(1)^n$ when $n$ tends to infinity  \cite{AN}. In
random environments, the asymptotic depends on whether the BPRE is
 strongly subcritical (SS), intermediate subcritical (IS) or
 weakly subcritical (WS) (see \cite{bpree} or the Preliminaries Section for details). A subcritical GW
 process is always strongly subcritical (SS).\\

In this paper, we study the role of the initial number of particles  in  such limit theorems.
 For a GW process, particles are independent. As a consequence, limit theorems
starting with several initial
particles derive from those for a single initial particle.  In random environment, particles do not reproduce  independently. Independence holds only conditionally on the
environment and  asymptotics may differ  from the GW case. \\

First,  we  determine  the dependence of  the asymptotic survival probability in terms of  the
initial number of particles. In that view, we  define
 $$\alpha_k:=\lim_{n\rightarrow\infty} \P_k(Z_n>0)/\P_1(Z_n>0).$$
 For a GW process, $\alpha_k=k$ and  the asymptotic survival
probability is proportional to the initial number of particles. This
equality still holds in the (SS+IS) case for BPRE, but not in the
(WS) case where a different asymptotic as $k\rightarrow\infty $  is established.
For the proof, we need an asymptotic result on random walks with
negative drift, which gives the sum  of the logarithms
of the mean number of offsprings for the  of successive environments. We refer to  \cite{Grey}
for asymptotics of the extinction  probability when the number of initial particles tends to infinity in the supercritical
case.\\

Moreover, when the BPRE is (SS) or (IS), if we condition on the survival of
the population at generation $n$, then 
 only one initial particle survives at generation $n$ when $n\rightarrow \infty$,  just as for a GW process. But this
does not hold in the (WS) case, as stated in Section 2.2. Thus,
(WS) BPRE conditioned to survive has  a supercritical behavior, as previously observed in
\cite{slide}. \\

In Section 3.3, we give an interpretation of these results in terms of environments. Conditioning 
on non-extinction induces a selection of environments with high reproduction law. In
the (SS+IS) case, we prove that the survival probability of the
branching process in the selected environments   is still  zero. This
is obvious  if environments are $a.s.$ subcritical, $i.e.$ $f'(1)<1$
$a.s.$ But in the (WS) case,
 conditioning
on the survival of the population selects only supercritical
environments, which means that  the sequence of selected  environments has
$a.s.$ a  positive survival probability. Finally
letting  the initial number of particles tend to infinity,  the
sequence of environments selected by conditioning on the survival of the population 
becomes subcritical again. \\

Finally, in Section 3.4, we consider the size of the population conditioned to survive and  we are  interested in the characterization of the Yaglom 
quasistationary distributions starting from $k$ particles :
$$\lim_{n\rightarrow\infty} \P_k(Z_n=i \ \vert \ Z_n>0) \qquad (i\geq 1).$$
In Section 3.5, we focus on the Q-process associated to the subcritical 
BPRE, which is defined for all $l_1,l_2, ...,l_n \in \N,$ by
$$\P_k(Y_1=l_1, ..., Y_n=l_n)=\lim_{p\rightarrow \infty} 
\P_k(Z_1=l_1,...,Z_n=l_n \ \vert \ Z_{n+p}>0 ).$$
See $\cite{AN}$ for details on the Q-process associated to  GW. Again, these results depend on 
the subcritical regime.

\section{Preliminaries}

We start by recalling some known results  for subcritical BPRE. Note that $s \in \RRR^+
\mapsto \E(f'(1)^s)$ is a  convex function and define $\gamma$ and $\alpha$ in $[0,1]$  such that
\be
\label{alpha}
\gamma:=\inf_{\theta \in [0,1]}\big\{ \E\big(f'(1)^\theta \big) \big\}=\E\big(f'(1)^\alpha \big).
\ee
From now on, we assume $\E(f'(1)\vert \log(f'(1)) \vert)<\infty$. Note that $0<\gamma<1$, $\gamma \leq \E(f'(1))$, and
$$\gamma=\E(f'(1)) \ \Leftrightarrow  \ \E\big(f'(1)\t{log}(f'(1))\big)\leq  0. $$
There are three subcritical regimes (see \cite{bpree}).
\begin{itemize}
\item[$\star$] The strongly subcritical case (SS), where $\E(f'(1)\t{log}(f'(1)))<0$. In this case, assuming further
$$\E(Z_1 \log^+(Z_1))<\infty,$$
then there exist $c,\alpha_k>0$ such that, as $n\rightarrow \infty$ :
\be
\label{equivun}
\P_k(Z_n>0)\sim c\alpha_k\E(f'(1))^n, \qquad \alpha_1=1.
\ee
\item[$\star$] The intermediate  subcritical case (IS), where
$\E(f'(1)\t{log}(f'(1)))=0$. In this case, assuming further
$$\E\big(f'(1)\log ^2(f'(1))\big)<\infty, \qquad \E\big([1+\log^-(f'(1))]f''(1)\big)<\infty,$$
then there exist $c,\alpha_k>0$ such that as $n\rightarrow \infty$ :
\be
\label{equivdeux}
\P_k(Z_n>0)\sim c \alpha_k n^{-1/2} \E(f'(1))^n, \qquad \alpha_1=1.
\ee
\item[$\star$] The weakly subcritical case (WS), where $0<\E(f'(1)\t{log}(f'(1)))<\infty$.
In this case, assuming further
$$\E(f''(1)/f'(1)^{1-\alpha})<\infty, \qquad \E(f''(1)/f'(1)^{2-\alpha})<\infty,$$
then there exist $c,\alpha_k>0$ such that as $n\rightarrow \infty$ :
\be
\label{equiv3}
\P_k(Z_n>0)\sim c \alpha_k n^{-3/2}\gamma^n, \qquad \alpha_1=1.
\ee
\end{itemize}
In the rest of the paper,  we  take  the integrability assumptions above for granted  for each  case. See \cite{Vat}
for asymptotics with  a weaker hypothesis
in the (IS) case. \\

It is also known that the process $Z_n$ starting from $k$ particles and conditioned
to be non zero converges to a finite positive random variable $\YY_k$, called
the Yaglom quasistationary distribution (see \cite{bpree}) :
$$\E_k\big(s^{Z_n}\vert \ Z_n>0)\stackrel{n\rightarrow\infty}{\longrightarrow} \E\big(s^{\YY_k}\big).$$
See Section 3.3 for discussions about $(\YY_k)_{k\in\N}$.
$\newline$
Actually, in \cite{bpree}, the result and the proof of the convergence is given  for $k=1$. It can
 be generalized to $k\geq 1$ with the following modifications. We borrow  Notations from \cite{bpree}
%\begin{displaymath}
$$f_{k,l}:= \left\{
\begin{array}{ll}  f_k \circ   f_{k+1} \circ \dots \circ    f_{l-1}, &  k<l \\
f_{k-1} \circ  f_{k-2} \circ  \dots \circ f_{l}, & k>l \\
\text{id}, &  k=l.
\end{array}
\right.
$$
%\end{displaymath}
Then $1-\E_k(s^{Z_n}\vert Z_n>0)=\E(1-f_{0,n}^k(s))/\P_k(Z_n>0)$.
Lemma 2.1 of \cite{bpree} still holds replacing $f_{0,n}$ by $f_{0,n}^k$ and
$\P(Z_n>0)$ by $\P_k(Z_n>0)$. Lemma 2.2  also still
holds and results of  Lemma 2.3 can now be stated as follows. By convexity of $x\in [0,1]\rightarrow x^k$
and $(f_n)_{n\in\N}$, for every $n\geq 0$, 
we have $a.s. \ \exp(-S_i)(1-f_{i,0}(s)^k)\leq 1 \ (0\leq s\leq 1)$, 
where  $S_i=\log(k)+\log(f_0'(1))+...+\log(f_{n-1}'(1))$. Moreover $\exp(-S_n)(1-f_{n,0}(s)^k)$
converges $a.s.$ as $n\rightarrow \infty$, which is a direct consequence of the convergence for $k=1$ given in Lemma 2.3 in \cite{bpree} (noting also that this implies $f_{n,0}(s)\rightarrow 0 \ a.s.$ as $n\rightarrow \infty$). 

$\newline$

Finally, we consider   the case where the generating functions of the reproduction laws are a.s.
linear fractional. Indeed in this case
the survival probability in a given environment can  be computed explicitly since linear fractional generating functions
are stable by composition. Specifically, we suppose that
 \be \label{lfc}
 f(s)=1-\frac{A}{1-B}+\frac{A s}{1-B s} \quad a.s. \qquad (0\leq s\leq 1),
 \ee
where $A,B$ are two r.v. such that $A
\in[0,1]$,  $B \in [0,1)$ and $A+B\leq 1$. In this case, setting  for every $i\in \N$,
$$P_i:=f_{n-i}'(1)...f_{n-1}'(1), \qquad (P_{0}=1),$$
 we have (see \cite{agr}, \cite{Guiv} or \cite{koz})
\be \label{explfc} \P_1(Z_n>0 \ \vert  \ f_0,..., \
f_{n-1})=1-F_n(0)=\bigg(1+\sum_{i=0}^{n-1}
\frac{f_{n-i-1}''(1)}{2f_{n-i-1}'(1)}P_i \bigg)^{-1}P_n. \ee 
Let us label by $i\in\N$ the  initial
particles  and denote by
$Z^{(i)}_n$ the number of descendants of particle $i$ at
generation $n$. As
conditionally on $(f_0, ...,f_{n-1})$, ($Z_n^{(i)}, \ i\geq 1$) is
an iid sequence, we  get \be \label{lfcm}
\P_k(Z_n^{(1)}>0,...,Z_n^{(k)}>0 \ \vert \ f_0,..., \ f_{n-1})=
\bigg(1+\sum_{i=0}^{n-1} \frac{f_{n-i-1}''(1)}{2f_{n-i-1}'(1)}P_i
\bigg)^{-k}P_n^k. \ee 

We can get now lower bounds for survival probabilities of a general BPRE by
 a coupling argument. We use  that  for every  probability
generating function $f_i$, we can
find a linear fractional probability
generating function $\w{f}_i$ such that for every $s\in [0,1]$, $\w{f}_i(s)\geq f_i(s)$, $\w{f}_i'(1)=f_i'(1)$,
 $\w{f}_i''(1)=2f_i''(1)$ (see \cite{Guiv} or \cite{koz}).
Then, $\w{F}_n(0)\geq F_n(0) \  a.s.$ ensures that
\be
\label{majavecun}
\P_1(Z_n>0 \ \vert  \ f_0,...,f_{n-1} )\geq \P_1(\w{Z}_n>0 \ \vert  \ \w{f}_0,...,\w{f}_{n-1}) \qquad a.s.
\ee
More generally, for every $k\geq 1$, 
\bea
&& \P_k(Z^{(1)}_n>0, \ Z^{(2)}_n>0, \ ..., \ Z^{(k)}_n>0 \ \vert  \ f_0,...,f_{n-1} )\nonumber \\
&=&(1-F_n(0))^k \nonumber \\
&\geq &  (1-\w{F}_n(0))^k \nonumber \\
\label{couplage}
&=&\P_k(\w{Z}^{(1)}_n>0, \ \w{Z}^{(2)}_n>0, \ ..., \ \w{Z}^{(k)}_n>0 \ \vert  \ \w{f}_0,...,\w{f}_{n-1}) \qquad a.s.
\eea
$\newline$

\section{Subcriticality starting from several  particles}

We specify here the asymptotics of  survival probabilities   starting with
$k$ particles. Then we determine how many initial particles survive conditionally on non extinction of particles and we characterize the sequence of environments which are selected by this conditioning.
Finally we consider the Yaglom
quasistationary distributions of $(Z_n)_{n\in\N}$ and  the associated Q-process. In the
(SS)  case, results are those expected, $i.e.$
they are  analogous to  those of a GW process.
In the (IS) case, results are different for the Yaglom
quasistationary distribution and the Q-process. In the
(WS) case, all results are  different. \\

Recall that we label by $i\in\N$ each particle of the initial
population  and denote by
$Z^{(i)}_n$ the number of descendants of particle $i$ at
generation $n$. Thus $(Z^{(i)}_n)_{n\in\N}$ are identically
distributed BPRE   ($i\in\N$), with common distribution
 $(Z_n)_{n\in\N}$ starting with one particle. Conditionally on the
environments, these processes are   independent :  for all $n,k, l_i$
$\in \N$,
 $$\P_k( Z_n^{(i)}=l_i, \ 1\leq i\leq k \ \vert \ f_0, ..., f_{n-1} )=
\Pi_{i=1}^k \P_1(Z_n=l_i \ \vert f_0, ..., f_{n-1}).$$
Moreover, under
$\P_k$, $(Z_n)_{n\in \N}$ is a.s. equal to $\big(\sum_{i=1}^{k} Z_n^{(i)}\big)_{n\in\N}.$
$\newline$

\subsection{Survival probabilities starting with several particles}

Note that $x\mapsto  \E\big(f'(1)^x\t{log}(f'(1))\big)$ increases with $x$.
\begin{Pte}  \label{equivalents}  For every $k\in\N^*$, \\

(i)  If $\E\big(f'(1)^k\emph{log}(f'(1))\big)<0$, then there exists $c_k>0$ such that
$$\P_k(Z^{(1)}_n>0, \ Z^{(2)}_n>0, \ ..., \ Z^{(k)}_n>0 )\stackrel{n\rightarrow \infty}{\sim}c_k \E(f'(1)^k)^{n}
     $$
 and   $\E(f'(1)^k)<\E(f'(1)^{k-1})<...<\E(f'(1))$.  \\

(ii) If $\E\big(f'(1)^k\emph{log}(f'(1))\big)=0$, then there exists $c_k>0$ such that
$$\P_k(Z^{(1)}_n>0, \ Z^{(2)}_n>0, \ ..., \ Z^{(k)}_n>0 )\stackrel{n\rightarrow \infty}{\sim}c_k n^{-1/2}\E(f'(1)^k)^n.
       $$

(iii) If $\E\big(f'(1)^k\emph{log}(f'(1))\big)>0$, then there exists $c_k>0$ such that
$$\P_k(Z^{(1)}_n>0, \ Z^{(2)}_n>0, \ ..., \ Z^{(k)}_n>0 )
\stackrel{n\rightarrow \infty}{\sim}c_k n^{-3/2} \w{\gamma}^n,$$
with $\w{\gamma}=\inf_{u \in \RRR^+} \{\E(f'(1)^u)\} \ \in \ (0,1)$ and $c=c_1\geq c_2\geq ...\geq c_k.$ \\

Moreover, in the (IS+WS) case, $\w{\gamma}=\gamma$. In the (SS) case, $\w{\gamma}<\gamma=\E(f'(1))$.
\end{Pte}

The proof is given in Section 4.1 and uses the case where the probability generating function $f$
is $a.s.$ linear fractional.

%\begin{Pte}
%\label{limicond}
%Conditionally on $Z^{(1)}_n>0$, $Z^{(1)}_n+Z^{(2)}_n+...+Z^{(k)}_n$ converges weakly
%as $n$ tends to infinity to $Y_k$ such that
%$$Y_k>0 \ \t{a.s}, \qquad  Y_1<_{stoch}Y_2<_{stoch}..<_{stoch}Y_k<_{stoch}...$$
%\end{Pte}
$\newline$

In the (SS+IS) case,  the asymptotic probability of survival of particles is
proportional to the number of initial particles, as stated below.
This is not surprising and well know for subcritical GW
process. But this does not hold in the (WS) case. Recall that
$\alpha_k$  is defined as
$\lim_{n\rightarrow\infty} \P_k(Z_n>0)/\P_1(Z_n>0)$.

\begin{Thm}
\label{asympta} In the (SS+IS) case, for every $k\in\N$, $\alpha_k=k$. \\

In the (WS) case, $\alpha_k \rightarrow  \infty$ as $k\rightarrow \infty$ and there exists $M_+>0$ such that
$$\alpha_k \leq  M_+ k^{\alpha}\log(k), \quad (k\geq 2), $$
% \int_0^1 [1-(1-x)^k]x^{-\alpha-1} u(-\log(x))\emph{d} x \sim k^{\alpha} \qquad (k\rightarrow \infty)$$
where $\alpha\in (0,1)$ is given by ($\ref{alpha}$). \\
Assuming further $\E(f'(1)^{1/2} \log (f'(1)))>0$ ($i.e.$ $\alpha<1/2$)
and that $f''(1)/f'(1)$ is bounded by a constant, there exists $M_{\_}>0$ such that
$$\alpha_k\geq M_{\_} k^{\alpha}\log(k), \quad (k\in\N).$$
\end{Thm}
One can naturally conjecture that the last result still holds for
$1/2\leq \alpha<1$. The proof also uses  the linear fractional case
where, conditionally on  the environments, the survival probability is related to
  a random walk whose jumps are the $\log$ of the means of the reproduction law of the environments. This is
why  we  need  to prove  a result about random walk with negative drift conditioned to be larger than $-x<0$
(see Appendix). One way to  generalize the last result of the theorem
above to the case $\E(f'(1)^{1/2} \log (f'(1)))>0$ ($i.e.$ $\alpha<1/2$)  would be to improve Lemma \ref{rw}.
$\newline$

\subsection{Survival of initial particles conditionally on non-extinction}

We  turn our attention to the number of  particles that survive when we condition on the survival
of the whole population of particles. More precisely, denote by $N_n$ the number of particles in
generation $0$ whose descendance is alive at generation $n$. That
is, starting with $k$  particles :
$$N_n:=\#\{1 \leq i \leq k \ : \ Z_n^{(i)}>0\}.$$
We have the following elementary consequence
of Proposition \ref{equivalents}.
\begin{Pte}
\label{deux}
In the (SS+IS) case, for every $k\geq 1$,
$$ \lim_{n\rightarrow\infty}
\P_k(N_n>1 \ \vert \ Z_n>0)=0.$$
In the (WS) case, for every $k\geq 1$,
$$\lim_{n\rightarrow\infty}\P_k( N_n=k \vert \ Z_n>0)>0.$$
\end{Pte}

Thus, for (SS+IS) BPRE, conditionally on the survival of the
population, only one initial particle survives, as for GW.
But for (WS) BPRE, several initial particles  survive with
positive probability. Forthcoming Theorem \ref{env} gives an
interpretation of this property in terms of selection of favorable
environments by conditioning on non-extinction. This result has an application to the branching model for cell division with parasite infection considered in \cite{vb}. In particular it ensures that the separation of descendances of parasites (see Section 6.3 in \cite{vb}) holds only in the (SS+IS) case.
In the same vein, we refer to  \cite{reduced} for results on the reduced  process associated with
subcritical  BPRE in the linear fractional case : In the (WS) case, the number of particles 
of the reduced process is not $a.s.$ equal to $1$ in the first generations. \\

We next consider the situation  when the number of initial particles tends to infinity
in the (WS) case. We shall see  that the number of initial particles which survive conditionally on non-extinction   is finite $a.s.$ but
not bounded.

\begin{Thm} \label{contrnb} In the (WS) case, assuming $\E(f'(1)^{1/2}\log (f'(1)))>0$ ($i.e.$ $\alpha<1/2$)
and $f''(1)/f'(1)$ is bounded by a constant, there exist $A_l \downarrow _{l\rightarrow\infty}  0$ such that
for all $k\geq l\geq 0$,
$$
\limsup_{n\rightarrow \infty} \P_k(N_n\geq l \  \vert \ Z_n>0)\leq A_l.
$$

Moreover, for every $l\in \N^*$,
$$\liminf_{k\rightarrow \infty}\liminf_{n\rightarrow\infty} \P_k(N_n= l \  \vert \ Z_n>0)>0.$$
\end{Thm}
Thus, under the conditions of the theorem, 
$$
\limsup_{k\rightarrow \infty}
\limsup_{n\rightarrow \infty} \P_k(N_n\geq l \  \vert \ Z_n>0)\leq A_l, \quad \t{with} \ A_l \downarrow _{l\rightarrow\infty}  0.
$$
$\newline$

\subsection{Selection of environments conditionally on non-extinction}
We characterize here  the sequence of environments which are selected by
conditioning on the survival of particles. \\

We  denote by  $\F$ the set of generating functions and
for every $\i _n=(g_0,\dots,g_{n-1})\in \F^n$, by
$Z_{\i_n}$ the value at generation $n$ of the   branching process in varying environment whose reproduction law at
generation $l\leq n-1$ has  generating function $g_l$. Thus, for every $k\geq 1$,
\be
\label{Z}
\E_k(s^{Z_{\i_n}})=[g_0 \circ g_1\circ
\dots \circ g_{n-1}(s)] ^k \qquad (0\leq s\leq 1).
\ee
Then we denote by $p(\u _n)$ the survival probability of
a particle in environment $\u _n$:
\be
\label{probaZ}
p(\u _n):=\P_1(Z_{\u _n}>0).
\ee
$\newline$

Denote by $\f_n$ the sequence of environments until time $n$, $i.e.$
$$\f_n:=(f_0,f_1,\dots,
f_{n-1}).$$
In the subcritical case, $p(\f _n)\rightarrow 0$ $a.s.$ as $n\rightarrow \infty$  since  $(Z_n)_{n\in\N}$ becomes extinct $a.s.$ Roughly speaking, the sequences of  environments have $a.s.$ zero survival
probability. In the (SS+IS) case, conditioning on the survival
of particles does not change this fact, but it does in the (WS) case, as we can guess using
 Proposition \ref{deux}.  Coming back to the model of plants in random weather,  the survival of flowers in the (SS+IS) case is due to the exceptional  reproduction of plants (despite the weather), whereas in the (WS) case it is due to nice weather (and regular reproduction of plants). \\
More precisely, we prove   that in the (WS) case, the sequence of environments which are
selected   by conditioning on $Z_n>0$ have $a.s.$ a positive
survival probability. Thus, they are 'supercritical'. In
$\cite{slide}$, the authors had already remarked this supercritical
behavior of the BPRE $(Z_n)_{n\in\N}$ in the (WS) case by giving
an analog of the Kesten-Stigum theorem, $i.e.$ the convergence
of $Z_n/m^n$. 
\begin{Thm} \label{env} In the (SS+IS) case, for all $k\geq 1$,  $\epsilon>0$,
$$\lim_{n\rightarrow\infty} \P_k(p(\f_n)\geq \epsilon \ \vert \ Z_n>0)=0.$$

In the (WS) case, for every $k\geq 1$,
$$\liminf_{n\rightarrow \infty} \P_k(p(\f_n)\geq \epsilon \ \vert \ Z_n>0) \stackrel{\epsilon \rightarrow 0+}{\longrightarrow }1.$$
\end{Thm}
$\newline$

This supercritical behavior in the (WS) case disappears as $k$
tends to infinity. That is, the survival probability of selected sequences of
environments tends to $0$ as the number of particles grows to
infinity.
\begin{Pte} \label{densi}
In the (WS) case, for every $\epsilon>0$,
$$\limsup_{n\rightarrow \infty} \P_k(p(\f_n)\geq \epsilon \ \vert \ Z_n>0)
\stackrel{k\rightarrow \infty}{\longrightarrow }0.$$
\end{Pte}
% On can naturally wonder if $\nu_k$ converges weakly as $k$ tends to infinity.
%$\newline$

In other words, conditionally on the survival of $Z_n$,
 the more initial particles there are, the less environments
need to be favorable to allow the survival of the population, and the less
likely it is for  a given particle  to survive. This explains  why letting
the number of initial particles tend to infinity does not make the
number of surviving initial particles
 tend to infinity, as stated in Theorem \ref{contrnb}.
$\newline$
%We can conjecture that $\nu_1(\d\u)/p(\u) \sim x^{-\alpha}$
\subsection{Yaglom quasistationary distributions}

We focus now on the Yaglom quasistationary distribution of
$(Z_n)_{n\in\N}$ (see Preliminaries  for existence and references).
For the GW process, this distribution does not depend on the
initial  number of particles and is characterized by  a functional
equation. This result still holds for (SS) BPRE. Indeed, starting
with several particles,  conditionally on the survival of one given
particle, the others become extinct (see Proposition \ref{deux}).
Recalling that in the (SS+IS) case, $\gamma=\E(f'(1))$, and
writing $p.g.f.$ for probability generating function,  we have the 
following statement.
\begin{Thm}
\label{quasista} 
For every $k\geq 1$,
the BPRE $Z_n$ starting from $k$ and conditioned to be positve
converges in distribution as $n\rightarrow \infty$ to a r.v. $\YY_k$, whose p.g.f.
 $G_k$ verifies
$$\E(G_k(f(s)))=\gamma G_k(s) + 1-\gamma \qquad (0\leq s\leq 1).$$
In the  (SS+IS) case,  the distribution of $\YY_k$ does not depend on $k$ . \\Moreover, in the (SS) case, 
the common p.g.f.  of $(\YY_k : k\geq 1)$ is the unique   p.g.f. $G$ 
  which satisfies the functional equation above and $G'(1)<\infty$.
\end{Thm}

In the (WS) case, we leave open the question of determining whether the
quasistationary distribution $\YY_k$ depends on the initial number $k$
of particles. We know that for every $k\geq 1$,  $G_k$
verifies the same functional equation given above but we do not
know if the solution is unique. 
Moreover, other observations also lead us to believe
 that quasistationary distributions $\YY_k$ might not depend on $k$.
 For example, we can prove that if $Z_1 \in \{0,1,N\}$ for some $N\in \N^*$, then $\YY_1\stackrel{d}{=}\YY_N$. \\
%$\YY_1\stackrel{d}{=}\YY_2\stackrel{d}{=}...\stackrel{d}{=}\YY_k\stackrel{d}{=}...$\\ \\
%Finally, using $(\ref{contrnb})$,  we  should be able to  dominate   %$(\Upsilon_k)_{k\in\N^*}$ stochastically  by a r.v. $\Upsilon _0$  :
%$$\forall k \in \N^*, \Upsilon _k \leq_{stoch} \Upsilon _0.$$
%We can thus conjecture that there is also one single quasistationary distribution in the (WS) case, that is,
%$$\YY_1\stackrel{d}{=}\YY_2\stackrel{d}{=}...\stackrel{d}{=}\YY_k\stackrel{d}{=}...$$
$\newline$

\subsection{Q-process associated with a BPRE}

The Q-process $(Y_n)_{n\in\N}$ is  the BPRE $(Z_n)_{n\in\N}$ conditioned to survive in  the distant
future. See \cite{AN} for details in the case of  GW processes. In the (SS) case, the Q-process converges in distribution to the size
biased Yaglom distribution, as for GW process and finer results have been obtained in \cite{afa2}. In the
(IS+WS) case, the Q-process is transient. That is,  the
population needs to grow largely in the first generations so that it can  survive.\\

Recall that
 for all $l_1,l_2, ...,l_n \in \N,$ by
$$\P_k(Y_1=l_1, ..., Y_n=l_n)=\lim_{p\rightarrow \infty} \P_k(Z_1=l_1,...,Z_n=l_n \vert Z_{n+p}>0 ).$$

\begin{Pte}
\label{Qprocess}
$\star$ In the (SS) case, for every $k\in\N^*$, for all $l_1,l_2,...,l_n \in\N$,
$$\P_k(Y_1=l_1, ..., Y_n=l_n)=[\E(f'(1))]^{-n}\frac{l_n}{k}\P_k(Z_1=l_1,...,Z_n=l_n).$$
Moreover $(Y_n)_{n\in\N}$  converges in distribution to the size biased Yaglom distribution.
$$ \forall \ l \geq 0, \qquad \P_k(Y_n=l)\stackrel{n\rightarrow\infty}{\longrightarrow} \frac{l\P(\Upsilon=l)}{\E(\Upsilon)}.$$

$\star$ In the (IS) case, for every $k\in\N^*$, for all $l_1,l_2,...,l_n \in\N$,
$$\P_k(Y_1=l_1, ..., Y_n=l_n)=\E(f'(1))^{-n}\frac{l_n}{k}\P_k(Z_1=l_1,...,Z_n=l_n).$$
Moreover $Y_n\rightarrow\infty$ in probability as
$n\rightarrow\infty$.

$\star$ In the (WS) case, for every $k\in\N^*$, for all $l_1,l_2,...,l_n \in\N$,
$$\P_k(Y_1=l_1, ...,Y_n=l_n)=\gamma^{-n}\frac{\alpha _{l_n}}{\alpha_k}\P_k(Z_1=l_1,...,Z_n=l_n).$$
Moreover $Y_n$ tends to infinity  $a.s.$
%$$\P_k(Y_n=l)\stackrel{n\rightarrow\infty}{\longrightarrow} 0.$$
\end{Pte}

We  focus now on the environments of the Q-process. We endow $\F$
with distance $d$ given by the infinity norm
$$ d(f,g)=\parallel f-g\parallel _{\infty}$$
and we denote   by $\B(\F)$ the Borel $\sigma$-field. \\

We introduce the probability $\nu_k$ on  $(\F^{\N},\B(\F)^{\otimes
\N})$ which gives the distribution of the environments when  the
BPRE $(Z_n)_{n\in\N}$ starting from $k$ particles is  conditioned to
survive. Using Kolomogorov Theorem, it can be specified by its
projection on $(\F^{p},\B(\F)^{\otimes p})$ for every $p \in\N$,
denoted by $\nu_{k \ \vert \F^p}$, \bea \label{mesenv} \nu_{k \
\vert \F^p}(\d\i_p)&:=&\lim_{n\rightarrow\infty} \P_k\big( \f _p
 \in  \d \i_p \vert Z_{n+p}>0 \big) \\
&=& \gamma^{-p}\P\big(\f _p \in  \d
\i_p\big)\sum_{l=1}^{\infty} \P_k(Z_{\i _p}=l) \frac{\alpha_l}{\alpha _k}, \nonumber
\eea with
$\f_p=(f_0,\dots,f_{p-1})$ and $\gamma=\E(f'(1))$ in the (SS+IS) case.
The limit is the weak limit of probabilities  on
$(\F^{p},\B(\F)^{\otimes p})$ (see \cite{bill} for definition and Section
4.5 for the proof), which we endow with the distance $d_p$ given by
\be
\label{dp}
d_p\big((g_0,\dots,g_{p-1}), (h_0,\dots,h_{p-1}) \big) =\sup\{ \parallel  g_i-h_i\parallel _{\infty} : 0\leq i\leq p-1\}.
\ee

For every $\u \in \F^{\N}$, we denote by  $\u{\vert n}$ the first
$n$ coordinates of $\u \in \F^{\N}$ and we introduce the survival
probability in environment $\u \in \F^{\N}$ :

$$p(\u)= \lim_{n\rightarrow \infty}  \P(Z_{\u{\vert n}}>0).$$
$\newline$
One can naturally conjecture an analog of Theorem \ref{env} and Proposition \ref{densi}. That is, for every $k\in\N^*$,
\Bea
&&\t{In the (SS+IS) case,} \quad \nu_k(\{\u \in \F^{\N} : p(\u)=0\})=1. \qquad \qquad \qquad  \qquad \\
&&\t{In the (WS) case,} \quad \nu_k(\{\u \in \F^{\N} : p(\u)>0\})=1 \quad \ \t{and}
 \quad \ \nu_k(p(\f)\in \d x)\stackrel{k\rightarrow\infty}{\Longrightarrow} \delta_0(\d x).
\Eea
% manque une domination de \nu_k^{[n)}
A  perspective is to characterize the tree of particles when we
condition by the survival of particles, $i.e.$ the tree of particles
of the Q-process. Informally, for GW  process, this gives a spine
with finite iid subtrees (see \cite{Geig,Lyons}). This fact still
holds in the (SS+IS) case but we will observe a 'multispine tree' in
the (WS) case. $\newline$

\section{Proofs}
Recall that $\f_n=(f_0,...f_{n-1})$  and set for every $n \in \N$,
$$X_n:= \log(f_n'(1)),  \qquad S_n:=\sum_{i=0}^{n-1} X_i \quad (S_0=0),$$
$$L_n:=\t{min}\{S_i \  : \ 1\leq i \leq n\}.$$
To get limit theorems starting from $k$ particles,  we will work conditionally on the environments so that particles reproduce independently.  Thus, we need 
to   control the asymptotic distribution of $p(\f_n)=\P_1(Z_n>0 \ \vert \ \f_n)$. Roughly speaking, we prove now that $p(\f_n)\approx  \exp(L_n) \ a.s.$ as $n\rightarrow \infty$. The proof relies on the fact that in the fractional linear case, we can compute the survival probability at time $n$ as a function of the random walk $(S_i, 1\leq i\leq n)$ (see Preliminaries). We use then a  result on  random walk with negative drift conditioned to be above $x<0$ given in Appendix to get the lower bound in the linear fractional case. The lower bound in the general case follows by a coupling argument, whereas the upper bound is easy.

\begin{Lem} \label{majpsur}For every $n\in\N$, we have 
$$p(\f_n)\leq \exp(L_n) \quad a.s.$$
Moreover if $\E(f'(1)^{1/2} \log (f'(1)))>0$ ($i.e.$ $0<\alpha<1/2$) and $f''(1)/f'(1)$ is bounded, then there exists $\mu\geq 1$ such that for all $n\in \N$ and $x\in (0,1]$,
$$ \P(p(\f_n) \geq x) \geq  \P(L_n \geq \log ( \mu x ) )/4.
$$
\end{Lem}
\begin{proof}
For the upper bound, note that all $n\in\N$ and  $\u_n\in\F^{n}$,  we have,
$$p(\u _ n)=\P_1(Z_{\u _ n}>0)\leq \E_1(Z_{\u _ n})= \Pi_{i=0}^{n-1} g_i'(1).$$
Thus $p(\f_n)\leq e^{S_n}$ $a.s.$ Adding that $p(\f _ n)$ decreases $a.s.$ ensures that
$$p(\f_n)\leq e^{L_n} \quad a.s.$$

For the lower bound, use  (\ref{majavecun}) and (\ref{explfc}) to get
$$p(\f_n)\geq p(\w{\f}_n)= \frac{\w{P}_n}{1+\sum_{i=0}^{n-1} \frac{\w{f}_{n-i-1}''(1)}{2\w{f}_{n-i-1}'(1)} \w{P}_i}
=\frac{P_n}{1+\sum_{i=0}^{n-1} \frac{f_{n-i-1}''(1)}{f_{n-i-1}'(1)} P_i} \quad \t{a.s.},$$
where $P_i:=f_{n-i}'(1)...f_{n-1}'(1) \ (P_{0}=1)$.  Define  
$$S_i':=\log(f_{n-i}'(1))+...+\log(f_{n-1}'(1)) \quad  (1\leq i\leq n), \qquad S_0'=0.$$
Then $P_i=\exp(S_i')$ and assuming  that   $C:=\big(1+\t{ess sup}( \frac{f''(1)}{f'(1)})\big)^{-1}>0$, we have  
$$ p(\f _n)\geq C\frac{e^{S_n'}}{2\sum_{i=0}^{n-1} e^{S_i'}}\geq \frac{C}{2}  \frac{e^{S_{n}'-\t{max}\{S_j' : 0\leq j\leq n\}}}{\sum_{i=0}^{n}
e^{S_i'-\t{max}\{S_j' : 0\leq j\leq n\}}} \ \ \t{a.s.}$$
Thus,
\be
\label{lienrw}
  p(\f _n)\geq \frac{C}{2}
\frac{e^{L_{n}}}{\sum_{i=0}^{n} e^{L_{n}-S_i}}. 
\ee
As $\alpha<1/2$, the  forthcoming Corollary \ref{beta} in Appendix (Section \ref{sectrw}) ensures  that there exists $\beta>0$ such that
for all $n\in\N$ and $x\in (0,1]$,
\Bea \P(p(\f_n) \geq x) &\geq & \P(L_{n} \geq  \log (2\beta x  /C) )
\P\big( \sum_{i=0}^{n} e^{L_{n}-S_i} \leq \beta \ \vert \  L_{n} \geq
\log (2\beta x  /C) \big)  \\
&\geq & \P(L_n \geq \log ( \mu x ) )/4,
\Eea
writing $\mu = \min(1,2 \beta/C)$. 
\end{proof}
$\newline$

\subsection{Proofs of Section 3.1}
First we give the proof of  Proposition  \ref{equivalents}, which is split into three parts. It follows the proof of Theorem 1.2 in \cite{Guiv}. Using also the Lemma above, we are then able to  prove Theorem \ref{asympta}.
$\newline$
\begin{proof}[Proof of Proposition \ref{equivalents} (i)]
We  follow the proof of Theorem 1.2 (a) in \cite{Guiv} and introduce the probability
$\w{\P}$ such that under $\w{\P}$, the environments still are iid and their law is given by
$$\w{\P}(f \in \d g)=\E(f'(1)^k)^{-1}g'(1)^{k} \P(f\in \d g).$$
Then, writing $P_n=f_{0}'(1)...f_{n-1}'(1)$ ($P_{0}=1$), we have
$$\P(Z_n^{(1)}>0,..., Z_n^{(k)}>0)=\E((1-F_n(0))^k)=\E(f'(1)^k)^n\w{\E}((1-F_n(0))/P_n)^k).$$
As  $\E(f'(1)^k\t{log}(f'(1)))<0$, then $\w{\E}(\t{log}( f'(1)))<0$ and
Theorem 5 in \cite{at} ensures that
$$C=\lim_{n\rightarrow\infty} \frac{1-F_n(0)}{P_n}$$
exists $\w{\P}$ $a.s.$ and belongs to $]0,1]$.
Thus, as $n\rightarrow \infty$,
$$ \P_k(Z_n^{(1)}>0,..., Z_n^{(k)}>0) \sim \E(f'(1)^k)^n  \w{\E}(C^k).$$
Add that $s\mapsto \E(f'(1)^s)$ decreases for $s\in [0,\alpha]$ and $k<\alpha$ to complete the proof, where  $\alpha$
is given by $(\ref{alpha})$.
\end{proof}
$\newline$

\begin{proof}[Proof of Proposition \ref{equivalents} (iii)]
We follow the proof of Theorem 1.2 (c) in \cite{Guiv}.\\  \\
STEP 1. First we consider the linear fractional case. In that case, by  (\ref{lfcm}),
%conditionally on $(f_0, ...,f_{n-1})$, ($Z_n^{(i)}, \ i\geq 1$) is an iid sequence, we  get
$$\P_k(Z_n^{(1)}>0,...,Z_n^{(k)}>0 \vert \ f_0,..., \ f_{n-1})=
\bigg(1+\sum_{i=0}^{n-1} \frac{f_{n-i-1}''(1)}{2f_{n-i-1}'(1)}P_i \bigg)^{-k}P_n^k.$$
Define $\w{\gamma}$  by
$$\w{\gamma}=\inf_{s\in\RRR^+}\big\{ \E\big(f'(1)^s \big) \big\}=
\E\big(f'(1)^{\w{\alpha}} \big),$$
where  $0<\w{\alpha}<k$ since $\E(f'(1)^k\t{log}(f'(1)))>0$.
Let $\P _{\w{\alpha}}$ be the probability given by
$$\P_{\w{\alpha}}(f \in \d g)=\w{\gamma}^{-1}g'(1)^{\w{\alpha}} \P(f \in \d g).$$
Then
$$\P_k(Z_n^{(1)}>0,..., Z_n^{(k)}>0)=\w{\gamma}^n \E_{\w{\alpha}}\bigg[
\bigg(1+\sum_{i=0}^{n-1} \frac{f_i''(1)}{2f_i'(1)}P_i \bigg)^{-k}P_n^{k-\w{\alpha}}\bigg].$$
As $\E_{\w{\alpha} }(\t{log}(f'(1)))=0$, we apply Theorem 2.1 in \cite{Guiv} with
$$\phi (x)=x^{k-\w{\alpha}}, \quad \psi (x)=(1+x)^{-k}, \quad 0<k-\w{\alpha}<k,$$
so there exists $c_k>0$ such that, as $n\rightarrow \infty$,
$$\P_k(Z_n^{(1)}>0,..., Z_n^{(k)}>0) \sim c_k \w{\gamma}^n n^{-3/2}.$$

STEP 2. For  the general case,  we can use STEP 1.
Indeed, by  (\ref{couplage}), there exists a  BPRE $(\w{Z}_n)_{n\in\N}$ such that  $\w{f}$ is $a.s.$ linear fractional,
 $\w{f}'(1)=f'(1)$ and
 $$\P_k(Z^{(1)}_n>0, \ Z^{(2)}_n>0, \ ..., \ Z^{(k)}_n>0)\geq \P_k(\w{Z}^{(1)}_n>0, \
\w{Z}^{(2)}_n>0, \ ..., \ \w{Z}^{(k)}_n>0).$$
By STEP 1, this yields the existence of $c_k(1)>0$ such that \be
\label{pos} \P_k(Z^{(1)}_n>0, \ Z^{(2)}_n>0, \ ..., \
Z^{(k)}_n>0)\geq c_k(1)\gamma^n n^{-3/2}. \ee

Note  that by inclusion-exclusion principle, we have
\be
\label{binom}
\P_k(Z_n>0)=\sum_{i=1}^k (-1)^{i+1}\binom{k}{i}\P(Z_n^{(1)}>0, ...,Z_n^{(i)}>0).
\ee
Moreover, (\ref{equiv3}) ensures the convergence of $ \gamma^{-n}n^{3/2}\P_1(Z_n>0)$
to $c\alpha_1$. By induction, it  gives  the convergence of
$$\gamma^{-n} n^{3/2} \P(Z^{(1)}_n>0, \ Z^{(2)}_n>0, \ ..., \ Z^{(k)}_n>0).$$
to a constant $c_k$, which is positive by (\ref{pos}). \\

To complete the proof note that $\gamma=\w{\gamma}$ iff $[\E(f'(1)^s) ]'(1)\geq 0$, $i.e.$ in the (IS+WS) case.
\end{proof}
$\newline$
\begin{proof}[Proof of Proposition \ref{equivalents} (ii)]
The proof is close to the previous one. First, we consider the linear fractional case
and
%$$\P(Z_n>0 \vert Z_0=1, \ f_0,..., \ f_{n-1})=1-F_n(0)=
%\bigg(1+\sum_{i=0}^{n-1} \frac{f_{n-i-1}''(1)}{2f_{n-i-1}'(1)}P_i \bigg)^{-1}P_n.$$
 the probability
$\w{\P}$ defined by
$$\w{\P}(f \in \d g)=\E(f'(1)^k)^{-1}g'(1)^k \P(f\in \d g).$$
Using again (\ref{lfcm}), we get then
$$\P(Z_n^{(1)}>0,..., Z_n^{(k)}>0)=\E(f'(1)^k)^n \w{\E}
\bigg[\bigg(1+\sum_{i=0}^{n-1} \frac{f_{n-i-1}''(1)}{2f_{i-i-1}'(1)}P_i \bigg)^{-k}\bigg].$$
As $\w{\E} (\log(f'(1))=0$, we can  use  again Theorem 2.1 in \cite{Guiv} and conclude in the linear fractional
case. \\

The general case can be proved following STEP 2 in the previous proof.
\end{proof}
$\newline$

%Recalling (\ref{defnun}) and that for every $p\in\N$,
%$$\nu_{1 \ \vert \F^{p}}^{(n)} \stackrel{n\rightarrow\infty}{\Longrightarrow} \nu_{1 \vert \F^{p}},$$
%we have
%\Bea
%\frac{\nu_1(\d\u)}{p(\u)}(p(\u ) \in \d x)&=& \frac{\nu_1(p(\u) \in \d x)}{x} \\
%&=& \lim_{n\rightarrow \infty} \frac{\P (p^{(n)}(\u) \in dx \ \vert \ Z_n>0 )}{x} \\
%&=& \lim_{n\rightarrow \infty} \frac{\P (p^{(n)}(\u ) \in dx \ \vert \ Z_n>0 )}{x} \\
%&=& \lim_{n\rightarrow \infty} \frac{\P (p^{(n)}(\u) \in dx, \ Z_n>0)}{\P(Z_n>0) x}\\
%&=& \lim_{n\rightarrow \infty} \frac{\P (p^{(n)}(\u) \in dx)}{\P(Z_n>0)} \\
%&\sim & \lim_{n\rightarrow \infty} \frac{\P (e^{L_n} \in dx)}{\P(Z_n>0)} \qquad (x\rightarrow 0) \\
%&\sim &  x^{-\alpha-1}\emph{d} x \qquad (x\rightarrow 0)
%\Eea
%using (\ref{asymptl}) and the first part of the lemma.

\begin{proof}[Proof of Theorem \ref{asympta}]
\underline{Computation of $\alpha_k$ in the (SS+IS) case.}
In the (SS+IS) case, Proposition \ref{deux} and (\ref{binom}) ensure that for every $k\in\N$,
$$\P_k(Z_n>0)\sim k\P_1(Z_n>0), \qquad (n\rightarrow \infty).$$
Then, $\alpha_k=k,$ which gives the first result. \\
$\newline$

\underline{Limit of $\alpha_k$ in the (WS) case.}
Note that $\P_1(Z_{p+n}>0)=\sum_{k=1}^{\infty} \P_1(Z_p=k)\P_k(Z_n>0)$. Then, \be
\label{equal}
 \frac{\P_1(Z_{p+n}>0)}{\P_1(Z_n>0)}=\sum_{k=1}^{\infty}
\P_1(Z_p=k)\frac{\P_k(Z_n>0)}{\P_1(Z_n>0)}.
\ee
First, $\P_k(\cup_{i=1}^k\{Z_n^{(i)}>0\})\leq \sum_{i=1}^k \P_k(Z_n^{(i)}>0)$, which gives
$$\P_k(Z_n>0)/\P_1(Z_n>0)\leq k.$$
Moreover  $\sum_{k=1}^{\infty} \P_1(Z_p=k)k= \E(Z_p)<\infty$ and $\P_k(Z_n>0)/\P_1(Z_n>0)\stackrel{n\rightarrow \infty}{\longrightarrow } \alpha_k,$ so by bounded convergence, we get
$$\sum_{k=1}^{\infty}
\P_1(Z_p=k)\frac{\P_k(Z_n>0)}{\P_1(Z_n>0)}  \stackrel{n\rightarrow \infty}{\longrightarrow } \sum_{k=1}^{\infty}
\P_1(Z_p=k)\alpha_k. $$
Then, using again (\ref{equiv3}),  letting $n \rightarrow \infty$ in $(\ref{equal})$ yields
$$ \gamma^{p}=  \sum_{k=1}^{\infty} \P_1(Z_p=k)\alpha_k.$$
Assuming that $(\alpha_k)_{k \in\N}$ is bounded by $A$ leads to
$$\gamma^p\leq A \P_1(Z_p>0).$$
Letting  $p\rightarrow \infty$ leads to a contradiction with  $(\ref{equiv3})$. Adding that $\alpha_k$ increases ensures
that $\alpha_k\rightarrow \infty$ as  $k \rightarrow \infty.$  \\ \\

\underline{Upper bound  of $\alpha_k$ in the (WS) case.} Using  the independence of the particles conditionally on the environments, we have 
$$\P_k(Z_n>0 \ \vert \ \f_n)=1-\P_1(Z_n=0 \ \vert \ \f_n)^k=1-(1-p(\f_n))^k.$$
This yields the following expressions for the survival probability starting from $k$ particles,
\bea
\P_k(Z_n>0)&=& \E(1-(1-p(\f_n))^k) 
= \label{Psurk}
k\int_{0}^1 (1-x)^{k-1} \P (p(\f_n) \geq  x) \d x.
\eea
So we 
can write
\bea
\alpha _k &=& 
%\lim_{n\rightarrow \infty} \P_k(Z_n>0)/\P_1(Z_n>0)  \nonumber \\
\label{alph}
  \lim_{n\rightarrow \infty}  k\int_{0}^1 (1-x)^{k-1} \frac{\P (p(\f_n) \geq  x)}{\P_1( Z_n>0)} \d x.
\eea
Using the first inequality of Lemma \ref{majpsur}, we have then
\Bea
\alpha_k &\leq &\limsup_{n\rightarrow \infty} k\int_0^1 (1-x)^{k-1}  \frac{\P (\exp(L_n) \geq x)}{\P_1( Z_n>0)} \d x \\
&\leq  & k. \limsup_{n\rightarrow\infty}
  \frac{n^{-3/2}\gamma^n}{\P_1(Z_n>0)}. \limsup_{n\rightarrow\infty}\int_0^1 (1-x)^{k-1}  \frac{\P (\exp(L_n) \geq x)}{n^{-3/2}\gamma^n}.
\Eea
By (\ref{minmaj}), we can use Fatou's Lemma and (\ref{min}) ensures that exists a linearly growing function $u$ 
such that
$$\alpha_k\leq
k\limsup_{n\rightarrow\infty}
  \frac{n^{-3/2}\gamma^n}{\P_1(Z_n>0)}.\int_{0}^1 (1-x)^{k-1}  x^{-\alpha} u(\log(1/x))\d x.$$
  Thus, using (\ref{equiv3}) and the fact that $u$ is linearly growing, there exists a constant $C>0$ such that
\be
\label{unomajalpha}
\alpha_k\leq C k\int_{0}^1 (1-x)^{k-1}x^{-\alpha}[1+ \log(1/x)] \d x.
\ee
Finally, splitting the integral at $1/k$ and using integration by  parts,
\Bea
\int_0^1  (1-x)^{k-1}x^{-\alpha}\log(1/x) \d x 
% &=&
% \int_0^{1/k} (1-x)^{k-1}\log(1/x) x^{-\alpha} \d x+
% \int_{1/k}^{1} (1-x)^{k-1}\log(1/x) x^{-\alpha} \d x\\
&\leq   &   \int_0^{1/k} x^{-\alpha} \log (1/x) \d x
+ k^{\alpha}\log(k) \int_{1/k}^{1} (1-x)^{k-1}  \d x \\
&\leq &[-\alpha+1]^{-1}\big( k^{\alpha-1}\log(k)+[-\alpha+1]^{-1} k^{\alpha-1}\big) +
 k^{\alpha-1}\log(k). 
%&\leq & M (1+\log(k))k^{\alpha},
\Eea
Similarly $\int_{0}^1 (1-x)^{k-1}x^{-\alpha} \d x\leq [1-\alpha]^{-1}k^{\alpha-1}+k^{\alpha-1}$. Then (\ref{unomajalpha}) ensures that
there exists $M_+>0$ such that for every $k>0$, $\alpha_k\leq M_{+}k^{\alpha}\log(k).$  \\

\underline{Lower bound of $\alpha_k$ in the (WS) case} assuming further $\E(f'^{1/2}(1) \log (f'(1)))>0$ ($i.e.$ $\alpha<1/2$)
and $f''(1)/f'(1)$ is bounded. \\

By (\ref{equiv3}) and the second part of Lemma \ref{majpsur}, there exists $\mu\geq 1$ such that for every $x\in (0,1]$,
$$ \liminf_{n\rightarrow \infty} \frac{\P(p(\f_n) \geq
x)}{\P_1(Z_n>0)} =\liminf_{n\rightarrow \infty} \frac{\gamma^n n^{-3/2}}{\P_1(Z_n>0)} \frac{\P(p(\f_n) \geq x)}{\gamma^n n^{-3/2}}\geq  c^{-1}\liminf_{n\rightarrow \infty} \frac{\P(L_n\geq \log(\mu x))}{\gamma^n n^{-3/2}}.$$
Using $(\ref{min})$ and the fact that $u$ grows linearly, there exists $D>0$ such that 
$$ \liminf_{n\rightarrow \infty} \frac{\P(p(\f_n) \geq
x)}{\P_1(Z_n>0)} \geq Dx^{-\alpha}\log(1/[x\mu]).$$
%&\geq & a\gamma^{-1} e^{-\alpha \log(x\beta \beta'^{-1})}u(\log(x\beta \beta'^{-1})) \\
%&\geq & b_{\delta}x^{-\delta},
%for some constant $b_{\delta}$.
By (\ref{alph}) and  Fatou's Lemma, 
\Bea
\alpha_k &\geq & D \int_0^{1} (1-x)^{k-1}x^{-\alpha}[\log(1/x)+\log(1/\mu)] \d x 
%& \geq &  D\int_0^{1/k} (1-(1-x)^k)[\log(1/x) - \log(\mu)]x^{-\alpha} \d x.
\Eea
For all $k\geq \mu^2$ and $x\in (0,1/k]$, $\log(1/x)\geq 2\log(\mu)$. So  for every $k\geq \mu^2$,
\Bea
\alpha_k &\geq &  
2^{-1}Dk\int_0^{1/k} (1-x)^{k-1} x^{-\alpha}\log(1/x) \d x  \\
&\geq & 2^{-1}D k\log(k) \int_0^{1/k}  x^{-\alpha} \d x,
\Eea
which ensures that there exists $M_{-}$ such that for every $k\geq 1$, $\alpha_k\geq M_{-}k^{\alpha}\log(k).$
%In the general case, the integrability assumptions in the (WS) case (see Introduction)
% ensure that
% $$\E(f''(1)/f'(1))<\infty.$$
% Then  for every $\theta>0$,
% $$\sum_{k\geq  1} \P(f''(1)/f'(1)\geq e^{k\theta})<\infty.$$
% Thus (supposer f_i'' pas haut sur f_i' bas) ????,
%  $$\sum_{k\geq  1} \P(f''_i(1)/f'_i(1)\geq e^{k\theta} \ \ vert \ L_n\geq -x)<\infty.$$
% We can then prove the analog of Corollary \ref{beta} for the quantity
% $$\sum_{i=0}^{n-1} \frac{f_{i}''(1)}{f_{i}'(1)}e^{L_n'-S_i'}.$$
% This completes the proof.
\end{proof}

\subsection{Proofs of Section 3.2}
\begin{proof}[Proof of Proposition \ref{deux}]
The first part ($i.e.$  the (SS+IS) case) follows
from
$$
\P_k(\exists i\ne j, \ 1\leq i,j\leq k, \ Z_n^{(i)}>0, \ Z_n^{(j)}>0 \ \vert  Z_n>0)
\leq \binom{k}{2} \frac{\P_2(Z_n^{(1)}>0, \ Z_n^{(2)}>0)}{\P_k(Z_n>0)}, $$
the asymptotics given by Proposition \ref{equivalents} (i-ii-iii) and equations
(\ref{equivun}) and (\ref{equivdeux}). The second  part
($i.e.$  the (WS) case) is directly derived from Proposition \ref{equivalents} (iii) and (\ref{equiv3}).
\end{proof}
$\newline$
\begin{proof}[Proof of Theorem \ref{contrnb}]
Denote by $N(\u_n)$ the number of initial particles which
survive until generation $n$ where the successive
reproduction laws are given by $\u_n$ ($i.e.$ conditionally on $\f_n=\u_n$). Then, for
all $1\leq l\leq k$,
\Bea
\P_k(N_n= l) &= &\int_{\F^n} \P(\f_n\in \d \u_n) \P_k(N(\u_n)=l)\\
&=&
\int_0^1 \P(p(\f_n) \in \d x) \binom{k}{l} x^l(1-x)^{k-l}.
%&\leq & \int_0^1 \P(e^{L_n}  \in \d x) \binom{k}{l} x^l(1-x)^{k-l}.
%&\leq & c/(1-\theta) \gamma^n n^{-3/2}\int_0^1 x^{\theta-1} \binom{k}{l} x^l(1-x)^{k-l}
\Eea Note that  $x \in [0,1] \mapsto  x^l(1-x)^{k-l}$ is positive,
increases on $[0,l/k]$ and decreases on $[l/k,1]$. \\

First, we prove
the upper bound. By Lemma  \ref{majpsur}, $p(\f_n)\leq \exp(L_n) \  a.s.$, so that
\Bea
&&\int_0^{1} \P(p(\f_n) \in \d x) x^l(1-x)^{k-l} \\
&=&\int_0^{1} \P(p(\f_n) \in \d x, \ \exp(L_n)\leq l/k) x^l(1-x)^{k-l} +
\int_0^{1} \P(p(\f_n) \in \d x, \ \exp(L_n)>l/k) x^l(1-x)^{k-l}  \\
&\leq & \int_0^{1} \P(\exp(L_n) \in \d x)  x^l(1-x)^{k-l} + \P(\exp(L_n) \in (l/k,1]) (l/k)^l(1-l/k)^{k-l}.
\Eea
By (\ref{min}),
$$\limsup_{n\rightarrow \infty}
 \frac{\P(\exp(L_n) \in (l/k,1])}{\gamma^n n^{-3/2}}\leq u(\log(k/l))(k/l)^{\alpha}.$$
Second, using again the variations of $x \in [0,1] \mapsto  x^l(1-x)^{k-l}$ and $(\ref{minmaj})$, we get
\Bea
&&\lim_{n\rightarrow \infty} \int_0^{1} \frac{\P(\exp(L_n) \in \d x)}{n^{-3/2}\gamma ^n}
  x^l(1-x)^{k-l} \\
&\leq &
\int_0^{l/k} \nu_+(\d x)  x^l(1-x)^{k-l} + \nu _+([l/k,1])(l/k)^l(1-l/k)^{k-l} \\
&\leq & c_+ \int_0^1 \log(1/x)x^{-\alpha-1} x^l(1-x)^{k-l} \d x + c_+(1+\int_{l/k}^1
\log(1/x)x^{-\alpha-1} \d x )(l/k)^l(1-l/k)^{k-l} \\
&\leq & c_+ \int_0^1 \log(1/x)x^{-\alpha-1} x^l(1-x)^{k-l} \d x + c_+\big(1+\log(k/l)
\frac{(k/l)^{\alpha}-1}{\alpha}\big)(l/k)^l(1-l/k)^{k-l}.
\Eea

Putting the three last inequalities together  and using $u(\log(k/l))\leq C(1+\log(k/l))$ for some $C>0$ ensures that
there exists $D>0$ such that
\Bea
&& \limsup_{n\rightarrow \infty}
\int_0^{1} \frac{\P(p(\f_n) \in \d x)}{n^{-3/2}\gamma ^n} x^l(1-x)^{k-l} \\
&\leq & c_+ \int_0^1 \log(1/x)x^{-\alpha-1} x^l(1-x)^{k-l} \d x + D(1+\log(k/l)(k/l)^{\alpha})(l/k)^l(1-l/k)^{k-l}.
\Eea
Moreover, denoting by  $B$  the Beta function, we have
\Bea
&&\int_0^1 \log(x)x^{-\alpha-1} x^l(1-x)^{k-l} \d x \\
&=& \int_0^{1/k}  \log(1/x) x^{l-\alpha-1}(1-x)^{k-l}\d x
+\int_{1/k}^1  \log(1/x) x^{l-\alpha-1}(1-x)^{k-l}\d x \\
&\leq & \int_0^{1/k}  \log(1/x) x^{l-\alpha-1}\d x + \log(k) \int_{1/k}^1  x^{l-\alpha-1}(1-x)^{k-l}\d x \\
&\leq & (l-\alpha)^{-1}\big[
 \log(k)k^{\alpha-l}+(l-\alpha)^{-1} k^{\alpha-l}\big] +  \log(k) B(l-\alpha,k-l+1),
\Eea
 by integration by parts. By Stirling's formula,  there exists
$C>0$, and then $C',C''>0$ such that for all $1\leq l\leq k$,
 \bea
 \binom{k}{l}k^{-\alpha} B(l-\alpha,k-l+1) &\leq & C
\frac{k^{k-\alpha+1/2}}{l^{l+1/2}(k-l)^{k-l+1/2}}
\frac{(l-\alpha)^{l-\alpha-1/2}(k-l+1)^{k-l+1/2}}{(k-\alpha+1)^{k-\alpha+1/2}} \nonumber \\
 &\leq &  C'\frac{(l-\alpha)^{l-\alpha-1/2}(k-l+1)^{k-l+1/2}}{l^{l+1/2}(k-l)^{k-l+1/2}} \nonumber  \\
\label{stir1}
 &\leq & C'' \frac{1}{ l^{1+\alpha}},
 \eea
where the last inequality comes from the fact that $(1/x+1/2)\log(1+x)$ is bounded for  $x \in [0,1]$, so that $(k-l+1/2)\log(1+1/(k-l))$ is bounded for $1\leq l < k$.

 Then, combining the three last inequalities gives \Bea
\limsup_{n\rightarrow \infty} \frac{\P_k(N_n=
l)}{k^{\alpha}\log(k)n^{-3/2}\gamma ^n}  &\leq &
\limsup_{n\rightarrow \infty} \frac{\binom{k}{l}}{k^{\alpha}\log(k)} \int_0^{1} \frac{\P(\exp(L_n) \in \d x)}{n^{-3/2}\gamma ^n}   x^l(1-x)^{k-l} \\
&\leq&
(l-\alpha)^{-1}\big[\binom{k}{l}k^{-l}+(l-\alpha)^{-1}
k^{-l}/\log(k)+ C'' \frac{1}{ l^{1+\alpha}} \big] \\
 &&+D\binom{k}{l}(k^{-\alpha}/\log(k)+l^{-\alpha})(l/k)^l(1-l/k)^{k-l}. \Eea
Adding that \be \label{stir2} \binom{k}{l} k^{-l}\leq 
\frac{1}{l!}, \ee  there exists $D'>0$ such that
$$\limsup_{n\rightarrow \infty} \frac{\P_k(N_n= l)}{k^{\alpha}\log(k)n^{-3/2}\gamma ^n}\leq
D'\big[ \frac{1}{ l^{1+\alpha}}+ \frac{1}{l!}
+\binom{k}{l}l^{-\alpha}(l/k)^l(1-l/k)^{k-l}\big]. $$ Then, \Bea
\limsup_{n\rightarrow\infty} \frac{\P_k(N_n\leq 
l)}{k^{\alpha}\log(k)n^{-3/2}\gamma ^n}
&=&\limsup_{n\rightarrow\infty} \sum_{l'=l}^k \frac{\P_k(N_n=
l')}{k^{\alpha}\log(k)n^{-3/2}\gamma ^n} \\
&=& \sum_{l'=l}^k \limsup_{n\rightarrow\infty}  \frac{\P_k(N_n=
l')}{k^{\alpha}\log(k)n^{-3/2}\gamma ^n} \\
&\leq & D \sum_{l'=l}^k \big[\frac{1}{ l'^{1+\alpha}}+ \frac{1}{l'!}
+\binom{k}{l'}l'^{-\alpha}(l'/k)^{l'}(1-l'/k)^{k-l'}\big] \\
&\leq&  D \bigg[\sum_{l'=l}^k \big[ \frac{1}{ l'^{1+\alpha}}+
\frac{1}{l'!}\big] + l'^{-\alpha}\bigg]
 \Eea
Recalling that $\P_k(Z_n>0)\sim c\alpha_k n^{-3/2}\gamma^n, \
(n\rightarrow\infty)$ and $\alpha_k\geq M_{\_} \log(k)k^{\alpha}, \
(k\in \N)$ (see Theorem \ref{asympta}), we have 
$$\limsup_{n\rightarrow\infty} \P_k(N_n\geq  l \ \vert \ Z_n>0)=\limsup_{n\rightarrow\infty} \frac{\P_k(N_n\geq  l)}{c\alpha_k n^{-3/2}\gamma ^n}
\leq (cM_{\_})^{-1}D \bigg[\sum_{l'=l}^k \big[ \frac{1}{ l^{1+\alpha}}+
\frac{1}{l!}\big] + l^{-\alpha}\bigg].$$
This gives the first inequality of the proposition with $A_l=(cM_{\_})^{-1}D \big[\sum_{l'=l}^{\infty} \big[ \frac{1}{ l^{1+\alpha}}+
\frac{1}{l!}\big] + l^{-\alpha}\big]$.
$\newline$

We can prove similarly the lower bound. By Lemma \ref{majpsur}, for
every $x>0$,
 $$\P(p(\f_n) \geq x) \geq \P(L_n \geq \log (x\mu))/4.$$
Then, using also $(\ref{majpsur})$, for all $0\leq  l <k $ and  $N>0$,
\Bea
\P(p(\f_n) \in [l/k,Nl/k[)&=&\P(p(\f_n)\geq l/k)-
\P(p(\f_n)\geq Nl/k) \\
&\geq &\P(L_n \geq \log (\mu l/k))/4 - \P(\exp(L_n)\geq Nl/k).
\Eea
By (\ref{min}) , we get
$$\liminf_{n\rightarrow \infty} \frac{\P(p(\f_n) \in [l/k,Nl/k[)}{n^{-3/2}\gamma^n} \geq (k/l)^{\alpha}[\mu^{-\alpha}u(\log (k)-\log(\mu l))/4 - N^{-\alpha} u(\log (k)-\log(Nl))].$$
Then, as $u$ is linearly growing, we can fix $N\geq 1$ so  that there exists $C>0$ such that
\be
\label{minorr}
\liminf_{k\rightarrow \infty}\liminf_{n\rightarrow \infty} 
\frac{\P(p(\f_n) \in [l/k,Nl/k[)}{k^{\alpha}\log(k)n^{-3/2}\gamma^n}\geq
l^{-\alpha}C.
\ee
Using that $$\P_k(N_n= l)= \int_0^1 \P(p(\f_n) \in \d x) \binom{k}{l} x^l(1-x)^{k-l},$$
and $x\rightarrow x^l(1-x)^{k-l}$ decreases on $[l/k,1]$, we have, for every $k\geq Nl$,
$$\P_k(N_n= l)\geq \P(p(\f_n) \in [l/k,Nl/k[)\binom{k}{l}(Nl/k)^l(1-Nl/k)^{k-l}.$$
Then $(\ref{minorr})$ and $\lim_{k\rightarrow \infty} \binom{k}{l}(Nl/k)^l(1-Nl/k)^{k-l}>0$ ensures that
$$\liminf_{k\rightarrow \infty}\liminf_{n\rightarrow\infty} \frac{\P_k(N_n=  l)}{k^{\alpha}\log(k) n^{-3/2}\gamma ^n}>0.$$
Use $\P_k(Z_n>0)\sim c \alpha_k n^{-3/2}\gamma^n$ and  the upper bound on $\alpha_k$ given in Theorem \ref{asympta} to conclude.
\end{proof}

\subsection{Proofs of Section 3.3}
\begin{proof}[Proof of Theorem \ref{env}]
Let us first consider the (WS+IS) case. Using that conditionally on $\f _n$,  $Z_n^{(1)}$ and $Z_n^{(2)}$ are independent,
$$\P_k(Z_n^{(1)}>0, \ Z_n^{(2)}>0)= \E( p(\f_n)^2).$$
Thus, for every $\epsilon>0$, by Markov inequality, 
$$\P_k(Z_n^{(1)}>0, \ Z_n^{(2)}>0 \ \vert \ Z_n>0) \geq \epsilon^2
 \P_k(p(\f_n) \geq \epsilon \ \vert \ Z_n>0).$$
By Proposition \ref{deux}, we get
$$\P_k(p(\f_n) \geq \epsilon \ \vert \ Z_n>0)
\stackrel{n\rightarrow \infty}{\longrightarrow }0.$$
$\newline$

In the (WS) case, by  (\ref{Psurk}),  for every $\epsilon \in (0,1]$ :
\Bea
\P_k(Z_n>0)
&\geq & \int_0^{\epsilon} \P( p(\f_n)\in \d x) (1-(1-x)^k).
\Eea
Moreover
\Bea
&&\big\vert\int_0^{\epsilon} \P( p(\f_n)\in \d x) (1-(1-x)^k)-\int_0^{\epsilon}
\P( p(\f_n)\in \d x)kx \big \vert \\
& & \qquad \leq  k \sup_{x\in[0,\epsilon [}\bigg\{\big\vert \frac{1-(1-x)^k}{kx}-1 \big\vert \bigg\}
\int_0^{\epsilon }\P( p(\f_n)\in \d x)x \\
& & \qquad \leq k \sup_{x\in[0,\epsilon [}\bigg\{\big\vert
 \frac{1-(1-x)^k}{kx}-1 \big\vert \bigg\} \P_1(Z_n>0).
\Eea
Putting these two inequalities together yields
$$\P_k(Z_n>0)\geq k\int_0^{\epsilon }
\P( p(\f_n)\in \d x)x -
k \sup_{x\in[0,\epsilon [}\bigg\{\big\vert \frac{1-(1-x)^k}{kx}-1 \big\vert \bigg\} \P_1(Z_n>0).$$
Then
\Bea
 \P_1(p(\f_n)\in [0,\epsilon )), \  Z_n>0)&=&\int_0^{\epsilon }
\P( p(\f_n)\in \d x)x \\
&\leq & \P_k(Z_n>0)/k +   \sup_{x\in[0,\epsilon [}
\bigg\{\big\vert \frac{1-(1-x)^k}{kx}-1 \big\vert \bigg\} \P_1(Z_n>0).
\Eea
Dividing by $\P_1(Z_n>0)$ and letting $n\rightarrow \infty$ ensure that
\Bea
 \limsup_{n\rightarrow\infty} \P_1(p(\f_n)\in [0,\epsilon ) \ \vert \ Z_n>0)
&\leq&
\limsup_{n\rightarrow\infty} \frac{\P_k(Z_n>0)}{k\P_1(Z_n>0)}+
\sup_{x\in[0,\epsilon [}\bigg\{\big\vert \frac{1-(1-x)^k}{kx}-1 \big\vert \bigg\} \\
&\leq&
\frac{\alpha_k}{k}+
\sup_{x\in[0,\epsilon)}\bigg\{\big\vert \frac{1-(1-x)^k}{kx}-1 \big\vert \bigg\}. \\
\Eea
Finally recall Theorem \ref{asympta} and use
$$\alpha_k/k \stackrel{k\rightarrow\infty}{\longrightarrow}0, \qquad \forall k \in \N^*, \ \
\sup_{x\in[0,\epsilon [}\bigg\{\big\vert \frac{1-(1-x)^k}{kx}-1 \big\vert \bigg\}
\stackrel{\epsilon \rightarrow 0}{\longrightarrow}0,$$
to get $\lim_{\epsilon \rightarrow  0+}\limsup_{n\rightarrow \infty}
\P_k(p(\f_n)\leq \epsilon \ \vert \ Z_n>0) =0.$
\end{proof}
$\newline$
\begin{proof}[Proof of Proposition \ref{densi}]
Recall that for every $\u_n\in\F^n$, $\P_k(Z_{\u_n}>0)=1-(1-p(\u_n))^k$. Thus,
\Bea
\P_k(p(\f_n) \in \d x \ \vert \ Z_n>0)&=&
\frac{\P(p(\f_n)\in \d x) (1-(1-x)^k)}{\P_k(Z_n>0)} \\
&=&
\P_1(p(\f_n) \in \d x \ \vert \ Z_n>0)
\frac{\P_1(Z_n>0)}{\P_k(Z_n>0)}\frac{(1-(1-x)^k)}{x}.
\Eea
Then, for every $\epsilon>0$,
\Bea
\limsup_{n\rightarrow\infty}
\P_k(p(\f_n) \geq \epsilon \ \vert \ Z_n>0)&=& \frac{1}{\alpha_k} \limsup_{n\rightarrow \infty}
 \int_{\epsilon}^1
\P_1(p(\f_n) \in \d x \ \vert \ Z_n>0)\frac{(1-(1-x)^k)}{x} \\
&\leq & \frac{1}{\epsilon \alpha_k},
\Eea
and the left hand part tends to zero as $k$ tends to
infinity by Theorem \ref{asympta}. This ends up the proof.
\end{proof}

\subsection{Proofs of section 3.4}
We know from Preliminaries that the BPRE  $(Z_n)_{n\geq 0}$ starting from
$k$ particles and conditioned to be positive converges in distribution to $\YY_k$, and we call $G_k$ its p.g.f :
$$G_k(s)=\lim_{n\rightarrow\infty} \E_k(s^{Z_n} \ \vert \ Z_n>0) \qquad (0\leq s \leq 1).$$
Adding that by \cite{bpree}, $G_1'(1)<\infty$ we can split the proof of Theorem $\ref{quasista}$ into three parts
\begin{itemize}
\item[(i)] For every $k\geq 1$,
$$\E(G_k(f(s)))=\gamma G_k(s) + 1-\gamma \quad (0\leq s\leq 1). $$
\item[(ii)] In the (SS+IS) case, for every $k\geq 1$, 
$\YY_k\stackrel{d}{=}\YY_1.$
\item[(iii)] There is a unique p.g.f $G$ which satisfies  
$$\E(G(f(s)))=\E(f'(1)) G(s) + 1-\E(f'(1)) \ \ (0\leq s\leq 1), \quad G'(1)<\infty  \qquad (\mathcal{E}).$$
\end{itemize}
One can note that (iii) proves (ii) in the (SS) case, adding that $G_k'(1)<\infty$ (whose proof for $k=1$
in \cite{bpree}  can be generalized). \\

\begin{proof}[Proof of (i)] Let $f_0$ be distributed as $f$ and independent of $(Z_n)_{n\in\N}$.
For every $n\in\N$,
\Bea
1-\E_k(s^{Z_{n+1}} \ \vert \ Z_{n+1}>0 ) &=& \frac{\E_k (1-s^{Z_{n+1}})}{\P_k(Z_{n+1}>0)} \\
&=&\frac{1}{\P_k(Z_{n+1}>0)} \sum_{i=1}^{\infty}\P_k(Z_n=i)\E_k(1-s^{Z_{n+1}} \ \vert \ Z_n=i) \\
&=& \frac{\P_k(Z_n>0)}{\P_k(Z_{n+1}>0)} \sum_{i=1}^{\infty} \P_k(Z_n=i \ \vert \ Z_n>0) \E(1-f_0(s)^i) \\
&=&\frac{\P_k(Z_n>0)}{\P_k(Z_{n+1}>0)}  \E_k(1- f_0^{Z_n}(s) \ \vert \ Z_n>0).
\Eea
Then letting $n$ tend to infinity and usingthe  asymptotics given in the Preliminaries section gives
$$1-G_k(s)=\gamma^{-1} \E(1-G_k(f_0(s)),$$
where $\gamma=\E(f'(1))$ in the (SS+IS) case.
\end{proof}
$\newline$
\begin{proof}[Proof of (ii)]
For every $i\geq 1$,
 $$\P_2(Z_n=i)= \P_2(Z_n^{(1)}=i, \ Z_n^{(2)}=0) + \P_2(Z_n^{(1)}=0, \ Z_n^{(2)}=i)+ \P_2(Z_n=i, \ Z_n^{(1)}>0, \ Z_n^{(2)}>0).$$
 Moreover $\vert \P_2(Z_n^{(1)}=i, \ Z_n^{(2)}=0) - \P_2(Z_n^{(1)}=i) \vert \leq \P_2(Z_n^{(1)}>0, \ Z_n^{(2)}>0)$, then
 $$\vert \P_2(Z_n=i)- 2\P_1(Z_n=i) \vert \leq 3\P_2(Z_n^{(1)}>0, \ Z_n^{(2)}>0).$$
 Thus, using Proposition \ref{deux},
$$\lim_{n\rightarrow \infty} \frac{\P_2(Z_n=i)}{\P_2(Z_n>0)} = \lim_{n\rightarrow \infty}
\frac{ 2\P_1(Z_n=i)}{{\P_2(Z_n>0)}}.$$
As $\alpha_2=\lim_{n\rightarrow\infty}
\P_2(Z_n>0)/\P_1(Z_n>0)=2$, we have
\Bea
\P(\YY_2=i)&=&\lim_{n\rightarrow \infty} \P_2(Z_n=i \ \vert \ Z_n>0) \\
&=&\lim_{n\rightarrow \infty} \frac{2\P_1(Z_n=i \ \vert \ Z_n>0)\P_1(Z_n>0)}{\P_2(Z_n>0)}  \\
&=& \P(\YY_1=i).
\Eea
Then $\YY_1\stackrel{d}{=}\YY_2$ and the same argument ensures
that for every $k\geq 1$, $\YY_k=\YY_1$.
\end{proof}
$\newline$

The proof of (iii) requires the following lemma 
\begin{Lem}
\label{lemfonct} If $H:[0,1]\rightarrow \RRR$ is a power series  continuous on $[0,1]$,  $H(1)=0$
and \be \label{equatder}
 H(s)=\frac{\E (H(f(s))f'(s))}{\E(f'(1))}, \quad (0\leq s\leq 1),
\ee then  $H\equiv 0$.
\end{Lem}
\begin{proof}
FIRST CASE : There exists $s_0 \in [0,1)$ such that $\E(f'(s_0))=\E(f'(1))$. \\

The monotonicity of $f'$ implies
$$f'(s_0)=f'(1) \quad a.s.,$$
and $f'$ is $a.s.$ constant on $[s_0,1]$. As it is a power series, $f'$ is $a.s.$ constant. \\
Thus
$$f(s)=f'(1)s+ (1-f'(1))\quad (0\leq s\leq 1), \qquad f'(1)\leq 1  \quad a.s.$$

Moreover, let  $\vert H(\alpha)\vert=\sup\{\vert H(s) \vert, \  s\in[0,1]\}$ with $\alpha \in [0,1)$, and note that
$$\E\big(f'(1)( H(\alpha)- H(f(\alpha)))\big)=0.$$
Thus $H(f(\alpha))=H(\alpha)$ $a.s.$ and by induction, recalling that $F_n=f_0\circ f_1\dots\circ f_{n-1}$, we have
$$H(F_n(\alpha))=H(\alpha) \quad \rm{a.s}.$$
As $Z_n$ is subcritical,  then $\E(F_n(\alpha))=\E(\alpha^{Z_n})\rightarrow 1$ as $n\rightarrow \infty$. So  $F_n(\alpha)\rightarrow 1$ in probability as $n\rightarrow \infty$. Adding that  $F_n(\alpha)<1 \ a.s.$ for every $n\in\N$
and that $H$ is a power series, then
$H$ is constant and equals  zero since $H(1)=0$. \\

SECOND CASE : For every $s_0 \in [0,1[, \ \E(f'(s_0))<\E(f'(1))$. \\
If $H\ne 0$ then  there exists $\alpha\in [0,1[$
such that
$$\t{sup}\{\mid H(s))\mid : s \in [0, \alpha]\}>0$$
Let $\alpha_n \in [\alpha,1[$ such that $\alpha_n\stackrel{n\rightarrow \infty }{\longrightarrow }1$.
Then,  for every $n\in\N$,
there exists $\beta_n \in [0,\alpha_n]$ such that :
\bea
\t{sup}\{\mid H(s)\mid : s \in [0, \alpha_n]\}&=& \mid H (\beta_n) \mid \nonumber \\
&\leq& \frac{\E(f'(\beta_n))}{\E(f'(1))} \sup\{\vert H(s) \vert, \ 0\leq s\leq 1\} \nonumber \\
%&\leq & \t{sup}\{\mid H(s)\mid : s \in [0, \t{max}(f_0(\alpha_n),f_1(\alpha_n))]\} \frac{f_0'(\beta_n)+f_1'(\beta_n)}{2m} \nonumber \\
&<& \t{sup}\{\mid H(s)\mid,  \ 0\leq s\leq 1\},
\nonumber \eea
since
$\t{sup}\{\mid H(s)\mid,  \ 0\leq s\leq 1\}>0$ and $\E(f'(\beta_n))<\E(f'(1))$.  As $I\cap
J=\varnothing$,  $\sup I <\sup (I\cup J) \Rightarrow \sup I<\sup
J$, we get
$$\t{sup}\{\mid H(s)\mid : s \in [0, \alpha_n]\}< \t{sup}\{\mid H(s)\mid : s \in ]\alpha_n, 1]\}.$$
And  $H(s) \stackrel{s\rightarrow 1}{\longrightarrow }0$  leads to a contradiction letting $n\rightarrow \infty$. So $H=0$.
\end{proof}

\begin{proof}[proof of (iii)]
Assume that $G_1$ and $G_2$ are two probability generating functions which verify $(\mathcal{E})$. By differentiation, $G_1'$ and $G_2'$ satisfy
$$\E( G'(f(s))f'(s))=\E(f'(1))G'(s).$$
Then $H:= G_2'(1)G_1' -G_1'(1)G_2'$ verifies the conditions of Lemma \ref{lemfonct}. As a consequence,
$$G_2'(1)G_1' =G_1'(1)G_2'.$$
And $G_1(0)=G_2(0)=0$, $G_2(1)=G_1(1)=1$ ensure that $G_1=G_2$, which gives the uniqueness for $(\mathcal{E})$.
\end{proof}

\subsection{Proof of Section 3.5}
\begin{proof}[Proof of Proposition \ref{Qprocess}] First, we have
\Bea
\P_k(Z_1=l_1,...,Z_n=l_n \vert Z_{n+p}>0 )&=&
\P_k(Z_1=l_1,...,Z_n=l_n) \frac{\P_{l_n} (Z_p>0)}{\P_k (Z_{n+p}>0)} .
\Eea
Then, using (\ref{equivun}), (\ref{equivdeux}), (\ref{equiv3}), we get
$$\lim_{p\rightarrow \infty}\P_k(Z_1=l_1,...,Z_n=l_n \vert Z_{n+p}>0 )=
\gamma^{-n}\frac{\alpha _{l_n}}{\alpha_k}\P_k(Z_1=l_1,...,Z_n=l_n).$$
and recall $\alpha_l=l$ in the (SS+IS) case to get the distribution of $(Y_n)_{n\in\N}$. \\

To get the limit distribution of $(Y_n)_{n\in\N}$, note that,  for every $l \in\N^*$,
$$\P_k(Y_n=l)=
\gamma^{-n}\frac{\alpha _{l}}{\alpha_k}\P_k(Z_n=l)
=\gamma^{-n} \P_k(Z_n>0) \frac{\alpha _{l}}{\alpha_k}\P_k(Z_n=l\ \vert  \ Z_n>0).$$
Use respectively
  (\ref{equivun}) and  (\ref{equivdeux}) to get the limit in distribution in the (SS) case and the (IS). \\

Finally, in the (WS) case, by (\ref{equiv3}),
there exists $C>0$ such that

$$
\P_k(Y_n\leq l) \leq  Cn^{-3/2}\frac{\alpha_l}{\alpha_k} \P_k(Z_n\leq l \ \vert \ Z_n>0) 
\leq C n^{-3/2}\frac{\alpha_l}{\alpha_k}.
$$
Then Borel-Cantelli Lemma ensures that $Y_n$ tends $a.s.$ to infinity as $n\rightarrow \infty$.
\end{proof}
$\newline$
\begin{proof}[Proof of (\ref{mesenv})]
To prove the convergence and the equality, note that
\Bea
\P_k \big( \f _p \in  \d \i _p \vert Z_{n+p}>0 \big)
&=& \frac{ \P\big(\f _p \in  \d \i_p\big) \E_k (\P_{Z_{\i _p}} (Z_{n}>0)) }{\P_k(Z_{n+p}>0)} \\
&=& \frac{\P_1(Z_n>0)}{\P_k(Z_{n+p}>0)} \sum_{l=1}^{\infty} \P_k(Z_{\i _p}=l) \frac{\P_l(Z_n>0)}{\P_1(Z_n>0)}.
\Eea
The asymptotics given in the Preliminaries section  ensure that
$$ \frac{\P_1(Z_n>0)}{\P_k(Z_{n+p}>0)} \stackrel{n\rightarrow \infty}{\longrightarrow} \frac{1}{\gamma^p \alpha_k},$$
and using the bounded convergence Theorem with
$$ \frac{\P_l(Z_n>0)}{\P_1(Z_n>0)} \stackrel{n\rightarrow \infty}{\longrightarrow} \alpha_l,
\quad  \frac{\P_l(Z_n>0)}{\P_1(Z_n>0)}\leq l, \quad \E(Z_{\u_p})<\infty.$$
ensures that
$$\lim_{n\rightarrow\infty} \P_k\big( \f _p
 \in  \d \i_p \vert Z_{n+p}>0 \big) \\
= \gamma^{-p}\P\big(\f _p \in  \d
\i_p\big)\sum_{l=1}^{\infty} \P_k(Z_{\i _p}=l) \frac{\alpha_l}{\alpha _k}.$$
This completes the proof.
\end{proof}
%\begin{proof}[Sketch of proof of Proposition \ref{limicond}]
%On définit une mesure d'environement arrière $\mu$. Notons $G_n(\i,s)$ la fonction génératrice
%le long de $\i$ cond à survivre  qui décroit quand $n$ tend vers l'infini.
%$$\E_k(s^{Z_n}\vert Z_n^{(1)) >0)=\int \mu (d\i) G(\i,s)$$
%\end{proof}
\section{Appendix : Random walk with negative drift}
\label{sectrw}
We study here the random walk $(S_n)_{n\in\N}$ with negative drift.
Indeed, in the linear fractional case, the survival probability 
is a functional of the random walk obtained by summing the successive means of environments (see (\ref{explfc})). In the general case, the random walk  appears in the lower bound of the survival probability (see (\ref{lienrw})). More precisely,  we need
to control the successive values of the random walk with negative drift conditioned to stay above $-x<0$.

 More specifically,
let $(X_i)_{i \in \N}$ be iid random variables distributed as $X$ with
$$ \E(X)<0.$$
We assume that for every $z\in [0,1]$, $\E(\exp(zX))<\infty$ and
$\E(X\exp(\alpha X))=0$ for some $0<\alpha<1$.
Set $\gamma:=\E(\exp(\alpha X))$,
$$S_n:=\sum_{i=0}^{n-1} X_i, \qquad (S_0=0),$$
and for all $n\in \N, \ k \in \N$,
$$L_n=\min\{S_i, \ 0\leq i\leq n\}.$$
Its asymptotic is given in Lemma 4.1 in $\cite{bpree}$ or Lemma  7 in
\cite{Hirano}. There exists a
linearly increasing  positive function $u$  such that, as $n\rightarrow \infty$
\be
\label{min}
\P(L_n\geq -x)\sim  e^{\alpha x}u(x)n^{-3/2}\gamma^n,
\ee
for $x\geq 0$ if the distribution $X$ is non-lattice, and for $x\in \lambda \Z$ if the distribution of
$X$ is supported by a centered lattice $\lambda \Z$. \\
Moreover for each $\theta >\alpha$, there exists $c_{\theta}>0$ such that
\be
\label{minmaj}
\P(L_n\geq -x)\leq c_{\theta} e^{\theta x}n^{-3/2}\gamma^n, \qquad (x\geq 0, \  n\in \N).
\ee
Finally, using (\ref{min}) and the fact that $u$ grows linearly, there exist $c_{-},c_+>0$ such that
the two following positive measures on $[0,1]$,
$$\nu_-(\d x)=c_{\_}\log(1/x)x^{-\alpha-1}\d x, \qquad
\nu_+(\d x)=c_+(\delta_1(\d x)+\log(1/x)x^{-\alpha-1}\d x ),$$
verify  for every $x \in ]0,1]$
\be
\label{bornes}
\nu _{-}([x,1])\leq
\lim_{n\rightarrow \infty} \frac{\P(e^{L_n}\geq x)}{n^{-3/2}\gamma^n}\leq \nu_{+}([x,1]).
\ee
$\newline$ 

We need to control the successive values of the random walk conditioned to stay above   $-x$
($x\geq 0$).
Under integrability conditions, it  is known that the process
$(S_{[nt]}/ n^{1/2}| L_n \geq 0)$ converges weakly to the  Brownian meander as $n\rightarrow\infty$ (see \cite{Iglehart}). Moreover Durrett \cite{Durrett}
has proved that if there exists $q>2$ such that $P\,\{X_1>x\}\sim x^{-q}L(x)$ as $x\rightarrow\infty$, where $L$ is slowly varying,
then  $(S_{[nt]}/ n| L_n \geq 0)$ converges   weakly to a non degenerate limit which has  a single jump. \\
We prove here that the random walk conditioned to stay above  $-x$
($x\geq 0$) spends very short time close to its minimum, by giving an upper bound of the number of visits to a level of the random walk reflected at its minimum. To be more specific, define
$$N_n(k)=\t{card}\{ i \in \N, \ i \leq n, \ k\leq S_i-L_n<k+1\}.$$

 \begin{Lem} \label{rw}
For every $\theta>\alpha$, there exists $d>0$ such that
$$\limsup_{n\rightarrow \infty} \P(N_n(k)\geq l \mid L_n\geq -x )
\leq d e^{\theta k}/ \sqrt{l}, \qquad (k,l \in \N, \  x\geq 0).$$
Moreover for all $\theta>\alpha$ and  $x\geq 0$, there exists $C>0$ such that
\be
\label{contr}
\P(N_n(k)\geq l \ \vert \ L_n\geq -x)\leq   C e^{\theta k}/\sqrt{l}, \quad (k,n, l \in \N).
\ee
\end{Lem}
$\newline$
%As a consequence, denoting by $(S_n^x)_{n\in\N}$ the randow walk $(S_n)_{n\in\N}$ conditioned
%to be larger than $x$, we have
%\begin{Cor}
%For every $x\leq 0$ and $f : \RRR \rightarrow \RRR^+$ such that
%$\sum_{k\in\Z, \ k\geq x} f(k)<\infty$,
%$$\E(\sum_{k\in\N} f(S_k^x))<infty$$
%\end{Cor}

Moreover, we will use the following consequence of the preceding lemma.
\begin{Cor} \label{beta} If $\alpha<1/2$, there exists $\beta > 0$  such that for
all $x\geq 0$ and $n\in\N$,
$$\P\big(\sum_{i=0}^{n} \exp(L_n-S_i) \leq \beta \  \vert \ L_n\geq -x \big)\geq 1/4.$$
\end{Cor}
$\newline$

For the sake of simplicity, we assume  that $X\in\Z$ $a.s.$ for the proof of Lemma \ref{rw}.
Thus
$$\forall k,n \in \N^2, \  \ N_n(k)=\t{card}\{ i \in \N, \ i \leq n, \ S_i-L_n=k\},$$
and we denote by $(T_j \ : \ 1\leq j\leq N_n(k))$ the successive times before
$n$ when $(S_i-L_n)_{i\in\N}$ visits $k$.  That is
$$T_1=\inf\{0\leq i\leq n \ : \  S_{i}-L_{n}=k\}, \quad
T_{j+1}=\inf\{T_j< i\leq n \ : \  S_{i}-L_{n}=k\}.$$
First, cutting the path of the random walk between two of these passage times
enables us to prove  the following result.
\begin{Lem}
\label{lemmtoutpourri}
If $X\in\Z$ $a.s.$, then for all $n,k,l,i$ and  $0\leq h\leq n$, we have $$ \P(L_n \geq  -i,\
N_n(k)\geq 2l, \
 T_{l}+n-T_{N_n(k)}=h)\leq (k+1)\P(L_{n-h}\geq -k)  \P(L_{h}\geq
 -i),$$
 and
 $$ \P(L_n \geq  -i,\  N_n(k)\geq 2l,
\
 T_{1}+n-T_{l}=h)\leq (k+1)\P(L_{n-h}\geq -k)  \P(L_{h}\geq
 -i).$$
\end{Lem}
\begin{proof}
We introduce  the first hiiting  time $M_n$ of the minimum $L_n$ before
time $n$ and $R_n(l)$ the last passage  time at $l$ before time $n$
$$M_n=\inf\{ j \in [1,n] : S_j=L_n\}, \qquad R_n(l):=\sup\{j\in [1,n] : S_j=l\}.$$

First, we consider the case where $M_n\in [0,T_l]\cup[T_{N_n(k)},n]$ and
split the path of the random walk between times $T_l$ and $T_{N_n(k)}$. For all $j\leq 0$,
 $k\geq 0$ and $0\leq n_1<n_2\leq n$, introduce then
\Bea
&& A(j,n_1,n_2)=\{L_n=j, \ N_n(k)\geq 2l, \ T_l=n_1, \ T_{N_n(k)}=n_2, \ M_n \in [0,n_1]\cup [n_2,n] \},\\
&& B(j,n_1,n_2)=\{\forall  m \in [1,n_1] : \ S_m\geq j, \ S_{n_1}=S_{n_2}=j+k, \\
&& \qquad \qquad \qquad \qquad  \forall  m\in [n_2+1, n] : S_m\geq j, \ S_m \ne j+k, \exists a \in  [0,n_1]\cup [n_2,n], \ S_a=j\}, \\
&& C(j,n_1,n_2)=\{\forall m \in [n_1,n_2 ] : S_m\geq j, \ S_{n_1}=S_{n_2}=j+k\}.
\Eea
Note that conditionally on $D(n_1,n_2):=\{S_{n_1}=S_{n_2}=j+k\}$, $B(j,n_1,n_2 )$ and $ C(j,n_1,n_2)$ are independent,
%\Bea
%\P(B(j,n_1,n_2)\cap C(j,n_1,n_2))&=& \P(S_{n_1}=S_{n_2}=j+k)\P(B(j,n_1,n_2) \ \vert \ S_{n_1}=S_{n_2}=j+k)\P(C(j,n_1,n_2) \ \vert \ S_{n_1}=S_{n_2}=j+k)
  $$\P( C(j,n_1,n_2) \ \vert \ S_{n_1}=j+k)\leq \P(L_{n_2-n_1}\geq -k),$$ and
$$A(j,n_1,n_2)\subset  B(j,n_1,n_2)\cap C(j,n_1,n_2).$$
Then, noting also that
$$ \P( C(j,n_1,n_2) \ \vert \ D(n_1,n_2))=\P(
C(j,n_1,n_2) \ \vert \ S_{n_1}=j+k)\P(S_{n_1}=j+k)/\P(D(n_1,n_2)),$$
we have
 \bea
\P(A(j,n_1,n_2))&\leq&\P(D(n_1,n_2))\P(B(j,n_1,n_2 ) \ \vert \ D(n_1,n_2))\P( C(j,n_1,n_2) \ \vert \ D(n_1,n_2)) \nonumber \\
%&=& \frac{\P(D(n_1,n_2))}{\P(S_{n_1}=j+k)\P(S_{n_2}=S_{n_1} \ \vert \ S_{n_1}=j+k)} \times \nonumber \\
%& &\P(S_{n_1}=j+k)\P(B(j,n_1,n_2 ) \ \vert \ D(n_1,n_2))\P( C(j,n_1,n_2) \ \vert \ S_{n_1}=j+k) \nonumber \\
&=& \P(S_{n_1}=j+k)\P(B(j,n_1,n_2 ) \ \vert \ D(n_1,n_2))\P( C(j,n_1,n_2) \ \vert \ S_{n_1}=j+k)\nonumber \\
\label{equn} &\leq & \P(L_{n_2-n_1}\geq -k)
\P(S_{n_1}=j+k)\P(B(j,n_1,n_2 ) \ \vert \ D(n_1,n_2)). \eea
Moreover, \Bea
&& \{L_n \geq  -i,\  N_n(k)\geq 2l, \  T_{l}+n-T_{N_n(k)}=h, \ M_n\in[0,T_{l}]\cup[T_{N_k(n)},n]\} \qquad \qquad \qquad \qquad  \qquad \qquad \qquad \qquad \\
&& \qquad \qquad \qquad \qquad \qquad \qquad \qquad \qquad \qquad
\qquad =  \bigcup_{\substack{j\geq -i,  \\  1\leq n_1< n_2\leq n,
\ n_1+n-n_2=h}} A(j,n_1,n_2). \Eea Then, using the last two
relations, \bea
& & \P(L_n \geq  -i,\  N_n(k)\geq 2l, \  T_{l}+n-T_{N_n(k)}=h, \ M_n\in[0,T_{l}]\cup[T_{N_k(n)},n]) \nonumber \\
&\leq  & \sum_{ \substack{j\geq -i,  \\ 1\leq n_1< n_2\leq n, \ n_2-n_1=n-h}} \P(A(j,n_1,n_2))
\nonumber \\
&\leq & \P(L_{n-h}\geq -k) \sum_{ \substack{j\geq -i, \\  1\leq n_1<
n_2\leq n, \ n_1+n-n_2=h}}\P(S_{n_1}=j+k)\P(B(j,n_1,n_2 ) \ \vert \
D(n_1,n_2))\nonumber. \eea Concatenating the path of the random walk
before time $n_1$ and after time $n_2$ gives
 \bea
& & \P(L_n \geq  -i,\  N_n(k)\geq 2l, \  T_{l}+n-T_{N_n(k)}=h, \ M_n\in[0,T_{l}]\cup[T_{N_k(n)},n]) \nonumber \\
&\leq & \P(L_{n-h}\geq -k) \sum_{ \substack{j\geq -i, \\  1\leq n_1< n_2\leq n, \ n_1+n-n_2=h}} \P(L_{n_1+n-n_2} =j, \ R_{n_1+n-n_2}(j+k)=n_1) \nonumber  \\
&\leq & \P(L_{n-h}\geq -k) \sum_{ j\geq -i}\P(L_h=j)   \nonumber \\
\label{majun} &=&  \P(L_{n-h}\geq -k)\P(L_h\geq -i). \eea

Second, we consider the case where $M_n\in [T_l,T_{N_n(k)}]$ and split the path of the random walk between times $T_1$ and $T_l$.
For all $j,j'\leq 0$,
 $k\geq 0$ and $0\leq n_1<n_2\leq n$, introduce then
 \Bea
&& A'(j,n_1,n_2)= \{L_n=-j, \ N_n(k)\geq 2l, \ T_l=n_1, \ T_{N_n(k)}=n_2, \ M_n \in [n_1,n_2] \},\\
&& B'(j,j',n_1,n_2) = \{\forall  m \in [1,n_1] : \ S_m\geq j', \ S_{n_1}=S_{n_2}=j+k, \\
&& \qquad \qquad \qquad \qquad  \forall  m\in ]n_2, n] : S_m\geq j', \ S_m \ne j+k, \exists a \in  [0,n_1]\cup [n_2,n] :
 \ S_a=j'\}, \\
&& C'(j,n_1,n_2)=\{\forall m \in [n_1,n_2 ] : S_m\geq j, \ S_{n_1}=S_{n_2}=k+j, \ \exists a \in [n_1,n_2] : S_a=j\}.
\Eea
Note that conditionally on $D(n_1,n_2)=\{S_{n_1}=S_{n_2}=j+k\}$, $B'(j,j',n_1,n_2 )$ and $ C'(j,n_1,n_2)$ are independent,
$$ A'(j,n_1,n_2)\subset \bigcup_{j'=j}^{j+k}  B'(j,j',n_1,n_2) \cap  C'(j,n_1,n_2)$$
and we get the analogue  of  $(\ref{equn})$,
$$\P(A'(j,n_1,n_2))
\leq  \sum_{j'=j}^{j+k}\P(L_{n_2-n_1}\geq -k)
\P(S_{n_1}=j+k)\P(B'(j,j',n_1,n_2 ) \ \vert \ D(n_1,n_2)).$$
Moreover
 \Bea
&&\{L_n \geq  -i,\  N_n(k)\geq 2l, \  T_{l}+n-T_{N_n(k)}=h, \ M\in [T_{l}, T_{N_k(n)}]\} \\
&&= \bigcup_{\substack{j\geq -i,  \\  1\leq n_1< n_2\leq n, \
n_1+n-n_2=h}} A'(j,n_1,n_2). \Eea
Then, following the proof of
(\ref{majun}), we get \bea
&&\P(L_n \geq  -i,\  N_n(k)\geq 2l, \  T_{l}+n-T_{N_n(k)}=h, \ M_n\in [T_{l}, T_{N_k(n)}]) \nonumber  \\
&\leq & \P(L_{n-h}\geq -k)  \sum_{\substack{j'\geq -i, \  j\in [j'-k,j'] }}
\sum_{ \substack{1\leq n_1< n_2\leq n, \\  n_1+n-n_2=h}} \P(S_{n_1}=j+k)\P(B'(j,j',n_1,n_2 ) \ \vert \ D(n_1,n_2)) \nonumber \\
&\leq &  \P(L_{n-h}\geq -k)  \sum_{\substack{j'\geq -i}} k   \max_{j\in [j'-k,j']}
\sum_{ \substack{1\leq n_1< n_2\leq n, \\  n_1+n-n_2=h}} \P(S_{n_1}=j+k)\P(B'(j,j',n_1,n_2 ) \ \vert \ D(n_1,n_2)) \nonumber \\
&\leq & \P(L_{n-h}\geq -k)  \sum_{j'\geq -i} k  \P(L_{h}=j')  \nonumber \\
\label{majdeux}
&\leq & k \P(L_{n-h}\geq -k)  \P(L_{h}\geq -i).
\eea

Combining the  inequalities (\ref{majun}) and (\ref{majdeux}), we  get
$$\P(L_n \geq  -i,\  N_n(k)\geq 2l, \
 T_{l}+n-T_{N_n(k)}=h)\leq (k+1)\P(L_{n-h}\geq -k)  \P(L_{h}\geq -i),$$
which proves the first inequality of the lemma. The second  can
be proved similarly concatenating the random walk between $[0,T_1]$ and
$[T_{N_n(k)},n]$.
\end{proof}
$\newline$
\begin{proof}[Proof of Lemma \ref{rw}]
Let $h\in \N$ such that $h\geq n/2$.  The first inequality of
 Lemma \ref{lemmtoutpourri} below ensures that
$$
\P(L_n \geq  -i,\  N_n(k)\geq 2l, \  T_{l}+n-T_{N_n(k)}=h) \leq  (k+1)\P(L_{h}\geq  -i) \P(L_{n-h}\geq -k).$$
Using (\ref{minmaj}),
$$
\P(L_n \geq  -i,\  N_n(k)\geq 2l, \  T_{l}+n-T_{N_n(k)}=h)  \leq  c_{\theta}(k+1)\P(L_{h}\geq  -i)e^{\theta
k}(n-h)^{-3/2}\gamma^{n-h}.$$
Moreover, using (\ref{min}), for every $i\in\N$,  there
exists $n_0\in\N$ such that for all $n_0/2\leq n/2\leq h$,
\be
\label{prem}
\P(L_{h}\geq  -i) \leq 2e^{i\alpha}u(i) h^{-3/2}\gamma ^{-h}\leq
2.2^{3/2}e^{i\alpha}u(i)n^{-3/2}\gamma ^{h}.
\ee
Then,  writing $c'_{\theta}=2.2^{3/2}.c_{\theta}$,
 \be
\label{premineg}
\P(L_n\geq -i,\  N_n(k)\geq 2l, \  T_{l}+n-T_{N_n(k)}=h)\leq
 c'_{\theta} e^{\alpha i}u(i)(k+1)e^{\theta k}\gamma^n
n^{-3/2}(n-h)^{-3/2}.
\ee

Similarly, for every $h$ such that $n_0/2\leq n/2\leq h$,  the second inequality of
  Lemma \ref{lemmtoutpourri} below ensures that
 \be
 \label{deuxineg}
 \P(L_n\geq -i,\  N_n(k)\geq 2l, \  T_{1}+n-T_{l}=h)
\leq c'_{\theta} e^{\alpha i}u(i)(k+1)e^{\theta k}\gamma^n n^{-3/2} (n-h)^{-3/2}.
\ee

Noting that $a.s.$
$$\{N_n(k)\geq 2l \}=\bigcup_{h=n/2}^{n-l}\{N_n(k)\geq 2l, \ T_{l}+n-T_{N_n(k)}=h\}
\bigcup_{h=n/2}^{n-l}\{N_n(k)\geq 2l, \ T_{1}+n-T_{l}=h\},$$
we can combine
 the  last two inequalities (\ref{premineg}) and (\ref{deuxineg}), which  gives for every $n\geq n_0$,
\Bea
 \P(L_n\geq -i, \  N_n(k)\geq 2l) &\leq & \sum_{n/2\leq h \leq n-l} \P(L_n\geq -i,\  N_n(k)
\geq 2l, \  T_{l}+n-T_{N_n(k)}=h) \\
 &&
+ \sum_{n/2\leq h \leq n-l} \P(L_n\geq -i,\  N_n(k)\geq 2l \  T_{1}+n-T_{l}=h)  \\
&\leq &  2c'_{\theta}e^{\alpha i}u(i)\gamma^n n^{-3/2}(k+1)e^{\theta k}
\sum_{ n/2\leq h \leq n-l } (n-h)^{-3/2} \\
&\leq &2c'_{\theta} e^{\alpha i}u(i)\gamma^n n^{-3/2}(k+1)e^{\theta k}\sum_{ h\geq l } h^{-3/2} \\
&\leq & 2.2c'_{\theta} e^{\alpha i}u(i)\gamma^n n^{-3/2}(k+1)e^{\theta
k}/\sqrt{l}, \quad (n\geq n_0). \Eea  Then,
using again ($\ref{min}$),
$$ \limsup_{n\rightarrow\infty} \P(L_n\geq -i, \  N_n(k)\geq 2l)/\P(L_n\geq -i)
\leq 4c'_{\theta}c_0^{-1} (k+1)e^{\theta k}/\sqrt{l}.$$
Using that $(k+1)e^{\theta k}=o(e^{\theta'k})$ if $\theta'>\theta$, this completes the proof of the first inequality of the lemma for $X\in\Z$. The general case can be
proved similarly. \\

Note that,  for every $\theta>\alpha$, when $h\geq n/2$, we can replace  (\ref{prem}) by
$$\P(L_{h}\geq  -i)\leq 2^{3/2}. c_{\theta}e^{\theta i}n^{-3/2}\gamma^h, \qquad (i,h,n \in \N).$$
Following the proof above ensures that there exists
$c_{\theta}''>0$ such  for all $i,  n, l \in \N $,
$$\P(L_n\geq -i, \  N_n(k)\geq 2l) \leq c_{\theta}'' e^{\theta i}\gamma^n n^{-3/2} e^{\theta k}/\sqrt{l}.$$
Thus, by (\ref{min}),  for every $x\geq 0$, there exists $C_x>0$ such that
$$
\P(N_n(k)\geq l \ \vert \ L_n\geq -x)\leq  2c_{\theta}'' C_x (k+1)e^{\theta k}/\sqrt{l}, \quad (k,n,l \in \N),
$$
which gives the second inequality of the lemma.
\end{proof}
$\newline$

\begin{proof}[Proof of  Corollary \ref{beta}]
Let $\alpha<1/2$ and $d>0$ given by Theorem \ref{asympta}. Fix $\alpha<\theta<\mu/2<1/2$. Choose also $k_0 \in \N$ such that
$$d \sum_{k\geq k_0}e^{(\theta-\mu /2)k}<1/2.$$

By (\ref{contr}), for every $x\geq 0$, there exists $D>0$ such that for every $n\in \N$,
 $$\P (N_n(k) \geq e^{\mu k} \ \vert \ L_n\geq -x)\leq D  e^{(\theta -\mu/2)k}$$
which is summable with respect to $k$. Thus, by Fatou's lemma,
$$\limsup_{n\rightarrow \infty}
\sum_{k\geq k_0} \P (N_n(k) \geq e^{\mu k} \ \vert \ L_n\geq -x)
\leq \sum_{k\geq k_0} \limsup_{n\rightarrow \infty} \P (N_n(k) \geq e^{\mu  k} \ \vert \ L_n\geq -x).$$
By Lemma \ref{rw}, this gives, for every $x>0$,
$$\limsup_{n\rightarrow \infty}
\sum_{k\geq k_0} \P (N_n(k) \geq e^{\mu k} \ \vert \ L_n\geq -x)
\leq d \sum_{k\geq k_0}e^{(\theta-\mu /2)k}.$$
Then,
$$\limsup_{n\rightarrow \infty} \P\big(\bigcup _{k\geq k_0} \{ N_n(k) \geq e^{\mu k}\} \ \vert
\ L_n\geq -x \big) < 1/2.$$
By Lemma \ref{rw} again, fix $N \in \N$ such that
$$\limsup_{n\rightarrow \infty} \P\big(\bigcup_{0\leq k <k_0} \{ N_n(k)  \geq N \}
\ \vert \ L_n\geq -x \big)\leq 1/4.$$
Then
$$\limsup_{n\rightarrow \infty} \P \big(\bigcup_{0\leq k <k_0} \{ N_n(k) \geq N \}
\bigcup _{k\geq k_0} \{ N_k(k)
\geq e^{\mu k}\}\ \vert \ L_n\geq -x \big) < 3/4.$$
Noting that
$$\sum_{i=0}^n
\exp(L_n-S_i)\leq \sum_{k=0}^{\infty} N_n(k)e^{-k},$$
this ensures that for every $x\geq 0$,
$$\liminf_{n\rightarrow \infty} \P\big(\sum_{i=0}^n
\exp(L_n-S_i) \leq  \beta  \ \vert \ L_n\geq -x  \big)> 1/4,$$
with $\beta:=\sum_{0\leq k <k_0} Ne^{-k+1}+ \sum_{k\geq k_0}e^{\mu k} e^{-k+1}.$ This gives the result.
\end{proof}
\section*{Acknowledgments}
I wish to thank Jean Bertoin, Amaury Lambert and Vladimir Vatutin for various suggestions, and the two anonymous referees of the previous draft for many corrections and improvements.

\end{document}